\documentclass[11pt]{article}
\usepackage{authblk} 

\usepackage[margin=1in]{geometry}

\usepackage[T1]{fontenc}
\usepackage[utf8]{inputenc}
\usepackage[english]{babel}
\usepackage{enumerate,amsmath,amsthm}
\usepackage{amssymb}
\usepackage{bbding}         
\usepackage{bm}             
\usepackage{bbm} 
\usepackage{bookman}
\usepackage{comment}
\usepackage{graphicx}       
\usepackage{wrapfig} 
\usepackage{fancyvrb}       
\usepackage{dcolumn}        
\usepackage{enumitem,etoolbox,hyperref}
\usepackage{tabu}
\usepackage{multirow}
\usepackage{graphicx} 
\usepackage{wrapfig} 
\usepackage[font=small,labelfont=bf]{caption} 
\graphicspath{ {Pictures/} } 
\usepackage[section]{placeins} 

\usepackage{hhline}

\usepackage{float}
\usepackage{tikz}
\usetikzlibrary{arrows}
\usetikzlibrary{arrows.meta}

\usepackage{array}
\usepackage{caption}

\newtheorem{definition}{Definition}

\usepackage{colortbl,xcolor}
\definecolor{mycolor}{RGB}{255, 240, 210} 

\theoremstyle{plain}
\newtheorem{theorem}{Theorem}
\newtheorem{lemma}[theorem]{Lemma}
\newtheorem{corollary}{Corollary}[theorem]
\newtheorem{proposition}[theorem]{Proposition}

\newtheorem*{no-lemma}{Lemma}

\theoremstyle{remark}
\newtheorem{remark}[theorem]{Remark}
\newtheorem{example}[theorem]{Example}

\newlength{\RoundedBoxWidth}
\newsavebox{\GrayRoundedBox}
\newenvironment{GrayBox}[1][\dimexpr\textwidth-4.5ex]%
   {\setlength{\RoundedBoxWidth}{\dimexpr#1}
    \begin{lrbox}{\GrayRoundedBox}
       \begin{minipage}{\RoundedBoxWidth}}%
   {   \end{minipage}
    \end{lrbox}
    \begin{center}
    \begin{tikzpicture}%
       \draw node[draw=black,fill=black!10,rounded corners,%
             inner sep=2ex,text width=\RoundedBoxWidth]%
             {\usebox{\GrayRoundedBox}};
    \end{tikzpicture}
    \end{center}}

\newcommand{\ru}[1]{\rule{0pt}{#1 em}}

\newcommand{\scal}{\cdot}

\makeatletter
\newcommand*\bigcdot{\mathpalette\bigcdot@{.5}}
\newcommand*\bigcdot@[2]{\mathbin{\vcenter{\hbox{\scalebox{#2}{$\m@th#1\bullet$}}}}}
\makeatother

\def\XXint#1#2#3{{\setbox0=\hbox{$#1{#2#3}{\int}$}
     \vcenter{\hbox{$#2#3$}}\kern-.5\wd0}}

\newcommand{\sgn}{\operatorname{sgn}}

\newcommand{\Exx}{\mathbb{E}\,}

\makeatletter
\newcommand{\doublewidetilde}[1]{{%
  \mathpalette\double@widetilde{#1}%
}}
\newcommand{\double@widetilde}[2]{%
  \sbox\z@{$\m@th#1\widetilde{#2}$}%
  \ht\z@=0.8\ht\z@
  \widetilde{\kern-0.2ex\box\z@}
}
\makeatother

\newcommand{\stkout}[1]{\ifmmode\text{\sout{\ensuremath{#1}}}\else\sout{#1}\fi}
\def\dout{\bgroup
 \markoverwith{\lower-0.2ex\hbox
 {\kern-.03em\vbox{\hrule width.2em\kern0.45ex\hrule}\kern-.03em}}%
 \ULon}
\MakeRobust\dout

\tikzset{
    vertex/.style={minimum size=1.5em},
    edge/.style={->,> = latex'},
    pair/.style={{Circle[length=4pt]}-{Circle[length=4pt]},
      shorten >=-2.5pt,
      shorten <=-2.5pt}
    }

\tikzset{
    vertex_gray/.style={minimum size=1.5em, fill=gray},
    edge_gray/.style={->, > = latex', draw=gray},
    pair_gray/.style={
      {Circle[length=4pt]}-{Circle[length=4pt]},
      shorten >=-3.5pt,
      shorten <=-3.5pt,
      draw=gray
    }
      }

\def\ssscal#1{
  \mathbin{\vcenter{\baselineskip.67ex
    \hbox{$\scal$}\hbox{$\scal$}\hbox{$\scal$}%
  }}%
}

\newcommand{\pairA}[1][1]{
        \node[vertex] (A1) at (1*#1,1) {};
        \node[vertex] (A2) at (1*#1,2) {};
        \node[vertex] (A3) at (1*#1,3) {};
        \node[vertex] (A4) at (1*#1,4) {};
        \node[vertex] (A5) at (1*#1,5) {};
        \node[vertex] (A6) at (1*#1,6) {};
}

\newcommand{\pairAB}[1][1]{
        \pairA[#1];
        \node[vertex] (B1) at (2*#1,1) {};
        \node[vertex] (B2) at (2*#1,2) {};
        \node[vertex] (B3) at (2*#1,3) {};
        \node[vertex] (B4) at (2*#1,4) {};
        \node[vertex] (B5) at (2*#1,5) {};
        \node[vertex] (B6) at (2*#1,6) {}
}

\newcommand{\pairABC}[1][1]{
        \pairAB[#1];
        \node[vertex] (C1) at (3*#1,1) {};
        \node[vertex] (C2) at (3*#1,2) {};
        \node[vertex] (C3) at (3*#1,3) {};
        \node[vertex] (C4) at (3*#1,4) {};
        \node[vertex] (C5) at (3*#1,5) {};
        \node[vertex] (C6) at (3*#1,6) {}
}

\newcommand{\pairABCD}[1][1]{
        \pairABC[#1];
        \node[vertex] (D1) at (4*#1,1) {};
        \node[vertex] (D2) at (4*#1,2) {};
        \node[vertex] (D3) at (4*#1,3) {};
        \node[vertex] (D4) at (4*#1,4) {};
        \node[vertex] (D5) at (4*#1,5) {};
        \node[vertex] (D6) at (4*#1,6) {};
}

\newcommand{\pairABCDE}[1][1]{
        \pairABCD[#1];
        \node[vertex] (E1) at (5*#1,1) {};
        \node[vertex] (E2) at (5*#1,2) {};
        \node[vertex] (E3) at (5*#1,3) {};
        \node[vertex] (E4) at (5*#1,4) {};
        \node[vertex] (E5) at (5*#1,5) {};
        \node[vertex] (E6) at (5*#1,6) {};
}

\newcommand{\pairABCDEF}[1][1]{
        \pairABCDE[#1];
        \node[vertex] (F1) at (6*#1,1) {};
        \node[vertex] (F2) at (6*#1,2) {};
        \node[vertex] (F3) at (6*#1,3) {};
        \node[vertex] (F4) at (6*#1,4) {};
        \node[vertex] (F5) at (6*#1,5) {};
        \node[vertex] (F6) at (6*#1,6) {};
}

\newcommand{\pairABCDEFGHI}[1][1]{
        \pairABCDEF[#1];
        \node[vertex] (G1) at (7*#1,1) {};
        \node[vertex] (G2) at (7*#1,2) {};
        \node[vertex] (G3) at (7*#1,3) {};
        \node[vertex] (G4) at (7*#1,4) {};
        \node[vertex] (G5) at (7*#1,5) {};
        \node[vertex] (G6) at (7*#1,6) {};
        \node[vertex] (H1) at (8*#1,1) {};
        \node[vertex] (H2) at (8*#1,2) {};
        \node[vertex] (H3) at (8*#1,3) {};
        \node[vertex] (H4) at (8*#1,4) {};
        \node[vertex] (H5) at (8*#1,5) {};
        \node[vertex] (H6) at (8*#1,6) {};
        \node[vertex] (I1) at (9*#1,1) {};
        \node[vertex] (I2) at (9*#1,2) {};
        \node[vertex] (I3) at (9*#1,3) {};
        \node[vertex] (I4) at (9*#1,4) {};
        \node[vertex] (I5) at (9*#1,5) {};
        \node[vertex] (I6) at (9*#1,6) {}
}

\newcommand{\oneincirc}{
\begin{tikzpicture}[baseline=-0.62ex, thick, main/.style = {draw,circle, inner sep = 2pt}, scale = 0.7]
    \node[main] (1) at (0,0) {$1$};
    \end{tikzpicture}
}

\newcommand{\fourcolab}{\begin{tikzpicture}[baseline = 5.1ex,scale = 0.37]
        \draw[fill=mycolor] (1.3-0.6, 0.5) rectangle (1*1.3+0.6, 5.5);
        \node[vertex] (A1) at (1*1.3,1) {};
        \node[vertex] (A2) at (1*1.3,1.7) {};
        \draw[fill=white] (1*1.3-0.35, 2.3) rectangle (1*1.3+0.35, 5.2);
        \draw[pair] (A1.center) to (A2.center);
        \end{tikzpicture}}

\newcommand{\fourcolabII}{\begin{tikzpicture}[baseline = 5.1ex,scale = 0.37]
        \draw[fill=mycolor] (1.3-0.6, 0.5) rectangle (1*1.3+0.6, 5.5);
        \node[vertex] (A1) at (1*1.3,2.7) {};
        \node[vertex] (A2) at (1*1.3,3.3) {};
        \draw[fill=white] (1*1.3-0.35, 3.9) rectangle (1*1.3+0.35, 5.2);
        \draw[fill=white] (1*1.3-0.35, 0.8) rectangle (1*1.3+0.35, 2.1);
        \draw[pair] (A1.center) to (A2.center);
        \end{tikzpicture}}

\newcommand{\fourcolabIII}{\begin{tikzpicture}[baseline = 5.1ex,scale = 0.37]
        \draw[fill=mycolor] (1.3-0.6, 0.5) rectangle (1*1.3+0.6, 5.5);
        \node[vertex] (A1) at (1*1.3,4.3) {};
        \node[vertex] (A2) at (1*1.3,5) {};
        \draw[fill=white] (1*1.3-0.35, 0.8) rectangle (1*1.3+0.35, 3.7);
        \draw[pair] (A1.center) to (A2.center);
        \end{tikzpicture}}

\usepackage{lipsum}

\title{Analysing the Moments of the Determinant of a Random Matrix Via Analytic Combinatorics of Permutation Tables}

\author{Dominik Beck\thanks{\href{mailto:beckd@karlin.mff.cuni.cz}{beckd@karlin.mff.cuni.cz}}}
\affil{\small Faculty of Mathematics and Physics, Charles University, Prague}
\author{Zelin Lv\thanks{\href{mailto:zlv@uchicago.edu}{zlv@uchicago.edu}}}
\author{Aaron Potechin\thanks{\href{mailto:potechin@uchicago.edu}{potechin@uchicago.edu}}}
\affil{\small The University of Chicago}

\usepackage[backend=bibtex]{biblatex}
\begin{document}

\maketitle

\begin{abstract}
We consider the following natural question. Given a matrix $A$ with i.i.d. random entries, what are the moments of the determinant of $A$? In other words, what is $\mathbb{E}[\det(A)^k]$? While there is a general expression for $\mathbb{E}[\det(A)^k]$ when the entries of $A$ are Gaussian, much less is known when the entries of $A$ have some other distribution.

In two recent papers, we answered this question for $k = 4$ when the entries of $A$ are drawn from an arbitrary distribution and for $k = 6$ when the entries of $A$ are drawn from a distribution which has mean $0$. These analyses used recurrence relations and were highly intricate. In this paper, we show how these analyses can be simplified considerably by using analytic combinatorics on permutation tables.


\vspace{1em}

\vspace{1em}
\begin{center}
    \textbf{Acknowledgement}
\end{center}

Dominik Beck was supported by the Charles University, project GA UK No. 71224 and by Charles University Research Centre program No. UNCE/24/SCI/022. Zelin Lv and Aaron Potechin were supported by NSF grant CCF-2008920.

\end{abstract}

\section{Introduction}
The determinant of a matrix is a fundamental quantity which is ubiquitous in linear algebra. A natural question about the determinant which is only partially understood is to determine the moments of the determinant of a matrix $A$ with i.i.d. random entries. In other words, given $k \in \mathbb{N}$, what is $ \mathbb{E}[\det(A)^k]$?

When the entries of $A$ are drawn from the normal distribution $\mathsf{N}(0,1)$, there is a general formula for $\mathbb{E}[\det(A)^k]$. In particular, Nyquist, Rice, and Riordan \cite{nyquist1954distribution} and later Pr\'{e}kopa \cite{prekopa1967random} independently showed that when $k$ is $even$, $\mathbb{E}[\det(A)^k] = \prod_{j=0}^{\frac{k}{2} - 1}{\frac{(n+2j)!}{(2j)!}}$.

That said, only a few results were previously known about $\mathbb{E}[\det(A)^k]$ when the entries of $A$ are drawn from a distribution which is not Gaussian.
\subsection{Prior work for non-Gaussian entries and our results}
In order to describe prior work and our results, we need a few definitions.

\begin{definition}
Given an $n \times n$ matrix $A$ with i.i.d. entries drawn from some distribution $\Omega$, we make the following definitions:
\begin{enumerate}
\item We define $m_r = \mathbb{E}[A_{ij}^r]$ to be the $r$th moment of the entries of $A$. When $m_1 \neq 0$, we define $B_{ij}=A_{ij}-m_1$ and we define $\mu_r = \mathbb{E}[B_{ij}^r]$ to be the $r$th centralized moment of the entries of $A$.
\item We define $f_k(n) = \mathbb{E}[\det(A)^k]$ to be the $k$th moment of the determinant.
\item We define $F_k(t) = \sum_{n=0}^{\infty}{\frac{t^n}{n!^2}f_k(n)}$ to be the generating function for the determinant.
\end{enumerate}
\end{definition}
\begin{example}\label{Ex:f42}
When $n=2$ and $k=4$, we have 
\begin{align*}
f_4(2) &= \Exx \left[\det(A)^4\right] = \Exx \left[(A_{11}A_{22}-A_{12}A_{21})^4\right] \\
&= \Exx \left[{A_{11}^4}A_{22}^4 \,-\,4{A_{11}^3}A_{22}^3{A_{12}}A_{21} \,+\, 6A_{11}^2{A_{22}^2}A_{12}^2{A_{21}^2} \,-\, 4A_{11}{A_{22}}A_{12}^3{A_{21}^3}\,+\, A_{12}^4{A_{21}^4}\right]\\ 
&= m_4^2-4m_3^2m_1^2+6m_2^4-4m_1^2m_3^2+m_4^2 = 2m_4^2 - 8m_3^2m_1^2+6m_2^4.
\end{align*}

\end{example}

\begin{example}
For $k=6$, when the entries of $A$ are drawn from $\mathsf{N}(0,1)$, we have $f_6(n) = \frac{n!(n+2)!(n+4)!}{48}$ and thus $F_6(t) = \sum_{n=0}^{\infty}{\frac{f_6(n)t^n}{n!^2}} = \sum_{n=0}^{\infty}{(n+1)(n+2)(n+4)!t^n}$. Note that $F_6(t)$ diverges everywhere except $t = 0$. That said, we can still treat it as a formal power series.
\end{example}
We now describe prior work on the moments of the determinant of a random matrix with i.i.d. entries which are not Gaussian and our results. For $k = 2$, Tur\'{a}n observed that when $m_1 = 0$ and $m_2 = 1$, $f_2(n) = n!$. More generally, Fortet \cite{fortet1951random} showed that $f_2(n) = n! (m_2 + m_1^2(n-1)) (m_2 - m_1^2)^{n-1}$ and $F_2(t) = (1+m_1^2t)e^{(m_2 - m_1^2)t}$. For $k = 4$, Nyquist, Rice, and Riordan \cite{nyquist1954distribution} showed that when $m_1 = 0$, $F_4(t) = e^{t(m_4-3 m_2^2)}/(1-m_2^2 t)^3$, which implies that when $m_1 = 0$, $f_4(n) = n!^2 \sum_{j=0}^n \frac{1}{j!} \left(m_4 - 3m_2^2\right)^j m_2^{2n-2j} \binom{n-j+2}{2}$. A special case of the fourth moment of the determinant of $\pm 1$ matrices was also independently given by Tur\'{a}n \cite{turan1955problem}.

Recently, the authors made progress on analyzing $\mathbb{E}[\det(A)^k]$ in two different ways. First, the first author \cite{beck2022fourth} generalised Nyquist, Rice, and Riordan's result for $k = 4$ to the setting where $m_1$ is arbitrary \cite{beck2022fourth}, showing that
\begin{equation*}
f_4(n) = n!^2 \sum _{w=0}^2 \sum _{s=0}^{4-2w} \sum _{c=0}^{n-s} \binom{4-2 w}{s}\frac{(1+c) m_1^{s+2 w} \mu_2^{2c-w} \mu _3^s \left(\mu _4-3\mu_2^2\right){}^{n-c-s}}{(n-c-s)!(2-w)! w!} d_w(c),
\end{equation*}
where $d_0(c) = 2+c$, $d_1(c) = c(2+c)$ and $d_2(c) = c^3$.

Second, we solved the case when $k = 6$ and $m_1 = 0$ \cite{LP2022}. In particular, we showed that for any distribution $\Omega$ over $\mathbb{R}$ such that $m_1 = 0$ and $m_2 = 1$,
\begin{align*}
f_6(n) = n!^2 \sum _{j=0}^n \sum _{i=0}^j \sum_{k=0}^{n-j} \frac{(1+i) (2+i) (4+i)! }{48 (n-j-k)!}\binom{10}{k}\binom{14+j+2i}{j-i} q_6^{n-j-k} q_4^{j-i}q_3^k,
\end{align*}
where $q_6 = m_6-10 m_3^2-15m_4 + 30$, $q_4 = m_4-3$, and $q_3 = m_3^2$.

Both of these analyses were technical and involved intricate recurrence relations. In this paper, we give a considerably simpler derivation of these results by using analytic combinatorics on permutation tables.

\subsection{Related work}
There are two natural variants of the question of determining the moments of the determinant of a random matrix which have been analysed. First, for \( \mathbb{E}[\det(A)^k] \) when \( A \) is symmetric or Hermitian, Zhurbenko solved the \( k=2 \) case for i.i.d. entries with mean 0 \cite{zhurbenko1968moments}, and later works extended this to Wigner matrices \cite{gotze2009second,mehta1998probability}. Recently, we generalised these results for \( k=2 \) to Hermitian \( A \) with i.i.d. entries above the diagonal with real expected values and i.i.d. diagonal entries \cite{Symmetric2024BLP}. Second, for \( \mathbb{E}[\det(U^\top U)^{k/2}] \) when \( U \) is an \( n \times p \) matrix with i.i.d. entries (\( p \leq n \)), Dembo solved the \( k=4 \) case for mean-0 entries \cite{dembo1989random}. This was recently generalised to $U$ with i.i.d. entries from an arbitrary distribution by the first author \cite{beck2022fourth}.

\section{Techniques}
\subsection{Analytic combinatorics}
We now give an overview of the analytic combinatorics techniques we use for our analyses. This overview is taken from our paper ``On the second moment of the determinant of random symmetric, Wigner, and Hermitian matrices'' \cite{Symmetric2024BLP}. We follow the notation of the textbook \emph{Analytic combinatorics} \cite{flajolet2009analytic} by Flajolet and Sedgewick.

Let $\mathcal{A}$ be a set of objects with a given structure where each $\alpha \in \mathcal{A}$ has a size $|\alpha| \in \mathbb{N} \cup \{0\}$ and a weight $w(\alpha) \in \mathbb{C}$ (which may be negative or even complex). We call $\mathcal{A}$ a combinatorial strucuture and view it in terms of the structure that its elements satisfy.

We say that $\mathcal{A}$ is labeled if each $\alpha \in \mathcal{A}$ is composed of atoms labeled by $[|\alpha|]= \{1,2,3,4,\ldots, |\alpha|\}$. Moreover, we assume that $\mathcal{A}_n = \{\alpha \in \mathcal{A}: |\alpha| = n \}$ is finite for all $n \geq 0$. We define $a_n = \sum_{\alpha \in \mathcal{A}: |\alpha| = n}{w(\alpha)}$ to be the total weight of the objects with size $n$.

Combinatorial structures can be composed together. One common composition is the \emph{star product}. Note that a tuple $(\alpha, \beta) \in \mathcal{A} \times \mathcal{B}$ cannot represent a labeled object of any structure as the atoms of $\alpha$ and $\beta$ are labeled by $[|\alpha|]$ and $[|\beta|]$, respectively. Relabeling our $\alpha$, $\beta$ as $\alpha'$, $ \beta'$ so that every number from $1$ to $|\alpha|+|\beta|$ appears once, we get a correctly labeled object. There are of course many ways how to relabel the objects. The canonical way is to use the \emph{star product}. We say $(\alpha' ,\beta') \in \alpha \star \beta$ if the new labels in both $\alpha'$ and $\beta'$ increase in the same order as in $\alpha$ and $\beta$ separately. An example is illustrated below.
\begin{figure}[htb!]
    \centering
    \begin{tikzpicture}
        \node at (-0.9, 0.5) {$\Bigg{(}$};

        \draw (0, 0) rectangle (1, 1);
        \draw[fill=black] (0,0) circle (2pt);
        \draw[fill=black] (0,1) circle (2pt);
        \draw[fill=black] (1,1) circle (2pt);
        \draw[fill=black] (1,0) circle (2pt);    
        \node at (0.5, 0.5) {\textcolor{red}{$\alpha'$}};
        \node at (-0.3, 1) {6};
        \node at (1.3, 1) {5};
        \node at (-0.3, 0) {2};
        \node at (1.3, 0) {4};

        \node at (1.8, 0.3) {$,$};

        \draw (2.5, 0) -- (3, 1);
        \draw[fill=black] (2.5,0) circle (2pt);
        \draw[fill=black] (3,1) circle (2pt);
        \node at (2.4, 0.5) {\textcolor{red}{$\beta'
$}};
        \node at (3.3, 1) {3};
        \node at (2.8, 0) {1};

        \node at (3.7, 0.5) {$\Bigg{)}$};

        \node at (4.3, 0.5) {$\in$};

        \begin{scope}[shift={(5, 0)}]
            \draw (0, 0) rectangle (1, 1);
            \draw[fill=black] (0,0) circle (2pt);
            \draw[fill=black] (0,1) circle (2pt);
            \draw[fill=black] (1,1) circle (2pt);
            \draw[fill=black] (1,0) circle (2pt);    
            \node at (0.5, 0.5) {\textcolor{red}{$\alpha$}};
            \node at (-0.3, 1) {4};
            \node at (1.3, 1) {3};
            \node at (-0.3, 0) {1};
            \node at (1.3, 0) {2};

            \node at (1.8, 0.5) {$\star$};

            \draw (2.5, 0) -- (3, 1);
            \draw[fill=black] (2.5,0) circle (2pt);
            \draw[fill=black] (3,1) circle (2pt);
            \node at (2.4, 0.5) {\textcolor{red}{$\beta$}};
            \node at (3.3, 1) {2};
            \node at (2.8, 0) {1};
        \end{scope}
    \end{tikzpicture}
    \label{fig:starprod}
\end{figure}

A key concept for labeled combinatorial structures is their exponential generating function (EGF for short) defined as $ A(t) = \sum_{n = 0}^{\infty}{\sum_{\alpha \in \mathcal{A}_n}{\frac{w(\alpha)t^n}{n!}}} = \sum_{n=0}^\infty{{a_n}\frac{t^n}{n!}}$. These exponential generating functions encode the relationships between combinatorial structures (i.e., how they are composed). We can write the following relations.

\bgroup
\renewcommand{\arraystretch}{1.3}
\begin{table}[tbh!]
    \centering
    \begin{tabular}{|c|c|c|c|}
    \hline
        $\mathcal{C}$ & representation of $\mathcal{C}$ & $w_\mathcal{C}(\gamma), \gamma \in \mathcal{C}$ & $C(t)$ \\
       \hline
       \ru{2} $\mathcal{A} + \mathcal{B}$ & $\mathcal{A}\cup \mathcal{B}$ & $\begin{cases}
    w_\mathcal{A}(\gamma) & \text{if } \gamma \in \mathcal{A}, \\
    w_\mathcal{B}(\gamma) & \text{if } \gamma \in \mathcal{B}
\end{cases}$ & $A(t) + B(t)$\\
       $\mathcal{A} \star \mathcal{B}$ &  $\{\gamma \mid \gamma \in \alpha \star \beta, \alpha \in \mathcal{A}, \beta \in \mathcal{B}\}$ & $w_\mathcal{A}(\alpha)w_\mathcal{B}(\beta)$ & $A(t) B(t)$\\
       $\textsc{Seq}(\mathcal{A})$ & $\sum_{k=0}^\infty\textsc{Seq}_k(\mathcal{A})$ & & $\frac{1}{1-A(t)}$\\
       $\textsc{Set}(\mathcal{A})$ & $\sum_{k=0}^\infty\textsc{Set}_k(\mathcal{A})$ & & $\exp{(A(t))}$\\
       $\textsc{Cyc}(\mathcal{A})$ & $\sum_{k=1}^\infty\textsc{Cyc}_k(\mathcal{A})$ & & $\ln \left(\frac{1}{1-A(t)}\right)$\\
    \hline
    \end{tabular}
    \label{tab:EGFancom}
\end{table}
\egroup

The meaning of $\textsc{Seq}_k(\mathcal{A})$, $\textsc{Set}_k(\mathcal{A})$, and $\textsc{Cyc}_k(\mathcal{A})$ is as follows.
\begin{itemize}
    \item $\textsc{Seq}_k(\mathcal{A})$ is shorthand for a \emph{sequence} and indeed it can be represented as (relabeled) $k$-tuples of objects taken from $\mathcal{A}$. Note that since everything is relabeled, even though $\alpha_i,\alpha_j$ might be the same for different $i, j$, the corresponding $\alpha_i'$, $\alpha_j'$ are always distinct. Formally, $\textsc{Seq}_k(\mathcal{A}) = \{(\alpha_1',\ldots,\alpha_k') \mid \alpha_i \in \mathcal{A}, i \in [k]\}$, where $(\alpha'_1,\ldots,\alpha'_k) \in \alpha_1 \star \cdots \star \alpha_k$.
    \item $\textsc{Set}_k(\mathcal{A})$ is a structure of \emph{sets} of $k$ (relabeled) objects from $\mathcal{A}$, that is, the order of the objects $\alpha_i'$ is irrelevant. Formally, $\textsc{Set}_k(\mathcal{A}) = \{\{\alpha_1',\ldots,\alpha_k'\} \mid \alpha_i \in \mathcal{A}, i \in [k]\}$. Alternatively, $\textsc{Set}_k(\mathcal{A})$ can be represented as the structure of classes of $k$-tuples in $\textsc{Seq}_k(\mathcal{A})$ which differ up to some permutation.
    \item $\textsc{Cyc}_k(\mathcal{A})$ represents the structure of classes of $k$-tuples in $\textsc{Seq}_k(\mathcal{A})$ which differ up to some cyclical permutation.
\end{itemize}

For completeness, we briefly explain these results. To see that the exponential generating function for $\mathcal{A} \star \mathcal{B}$ is $A(t) B(t)$, let $a_n$, $b_n$, and $c_n$ be the total weight of the objects of size $n$ in $\mathcal{A}$, $\mathcal{B}$, and $\mathcal{A} \star \mathcal{B}$ respectively. We have that $c_n = \sum_{j=0}^{n}{\binom{n}{j}{a_j}b_{n-j}}$, so $C(t) = \sum_{n=0}^{\infty}{\frac{{c_n}t^n}{n!}} = \sum_{n=0}^{\infty}{\sum_{j=0}^{n}{\frac{{a_j}t^{j}}{j!} \cdot \frac{b_{n-j}t^{n-j}}{(n-j)!}}} = A(t) B(t)$.
The generating functions for $\textsc{Seq}(\mathcal{A})$, $\textsc{Set}(\mathcal{A})$, and $\textsc{Cyc}(\mathcal{A})$ come from the Taylor series $\frac{1}{1-x} = \sum_{k=0}^{\infty}{x^{k}}$, $e^{x} = \sum_{k=0}^{\infty}{\frac{x^{k}}{k!}}$, and $-\ln(1-x) = \sum_{k=1}^{\infty}{\frac{x^k}{k}}$ respectively.

\begin{example}
Let $\mathcal{D}$ be the combinatorial structure of all derangements (i.e., permutations of $[n]$ where no element is mapped to itself). Any derangement can be decomposed into cycles of length at least two. Attaching a tag $u$ to each cycle so that the weight of each derangement is equal to $u^{\# \text{ of cycles}}$, we get a structure $\mathcal{D}_u$ which can be also constructed as $\mathcal{D}_u = \textsc{Set}\big{(}u\,\textsc{Cyc}_{\geq 2}\big{(}\, \oneincirc \,\big{)}\big{)}$. In terms of generating functions, this translates to $D_u(t) = \exp(-u t - u\ln(1-t)) = e^{-ut}/(1-t)^u$.
\end{example}


\subsection{Permutation tables}
As is standard, we denote the set of all permutations of $[n] = \{1,\ldots,n\}$ by $S_n$ and we define the sign $\sgn(\pi)$ of a permutation $\pi \in S_n$ to be $1$ if $\pi$ is the product of an even number of transpositions and $-1$ if $\pi$ is the product of an odd number of transpositions.
\begin{definition}
We say $\tau$ is a permutation ($k$-)table with $n$ columns 
if $\tau$ is a $k \times n$ table and there are permutations $\pi_1,\ldots,\pi_k \in S_n$ such that the $j$th row of $\tau$ is $(\pi_j(1),\ldots,\pi_j(n))$. 
We denote the set of all such $k$-tables with $n$ columns by $F_{k,n}$. Equivalently, $F_{k,n} = S_n \times \cdots\times S_n$ ($k$ copies of $S_n$). Finally, we denote $k$-tables with any number of columns by $F_k$. 
That is, structurally, $F_k = \cup_{n=0}^\infty F_{k,n}$.
\end{definition}
\begin{definition}
We define the sign of a table as the product of the signs of the permutations in its rows.
\end{definition}
\begin{definition}
We define the weight of the $i$-th column of $\tau \in F_k$ to be the expectation $\Exx \left[\prod_{j=1}^k A_{i\pi_j(i)}\right]$. We define the weight $w(\tau)$ of the whole table $\tau$ to be the product of the weights of its columns.
\end{definition}

\begin{example}
The following example in Figure \ref{fig:exatablf} shows a permutation table $\tau \in  F_{4,9}$ with weight $w(\tau) = m_1^{12}m_2^{7}m_3^2m_4$. The weight of each individual column is shown below each column. For instance, the second column corresponds to the term $A_{26}A_{22}A_{26}A_{23}$, whose expectation is ${m_1^2}m_2$.
\renewcommand{\arraystretch}{0.84}
\begin{figure}[H]
\centering
    \begin{tabular}{|
      >{\columncolor{mycolor}\centering}p{2.2em}|
      >{\columncolor{mycolor}\centering}p{2.2em}|
      >{\columncolor{mycolor}\centering}p{2.2em}|
      >{\columncolor{mycolor}\centering}p{2.2em}|
      >{\columncolor{mycolor}\centering}p{2.2em}|
      >{\columncolor{mycolor}\centering}p{2.2em}|
      >{\columncolor{mycolor}\centering}p{2.2em}|
      >{\columncolor{mycolor}\centering}p{2.2em}|
      >{\columncolor{mycolor}\centering}p{2.2em}|
      >{\centering\arraybackslash}p{1.5em}
      }
    \hhline{|=========|~}
    \ru{0.9}1 & 6 & 3 & 9 & 5 & 2 & 7 & 8 & 4 & $+$ \\
        3 & 2 & 1 & 9 & 4 & 6 & 7 & 5 & 8 & $+$ \\
        4 & 6 & 1 & 9 & 3 & 2 & 7 & 5 & 8 & $+$ \\ 
        2 & 3 & 1 & 5 & 4 & 6 & 7 & 8 & 9 & $-$\\   
        \hhline{|=========|~}
        \multicolumn{1}{c}{\ru{1.2}
        $m_1^4$} & 
        \multicolumn{1}{c}{\!\!$m_1^2m_2$\!\!} & 
        \multicolumn{1}{c}{\!\!$m_1m_3$\!\!} & 
        \multicolumn{1}{c}{\!\!$m_1m_3$\!\!} & 
        \multicolumn{1}{c}{\!\!$m_1^2m_2$\!\!} & 
        \multicolumn{1}{c}{\!\!$m_2^2$\!\!} & 
        \multicolumn{1}{c}{\!\!$m_4$\!\!} & 
        \multicolumn{1}{c}{\!\!$m_2^2$\!\!} & 
        \multicolumn{1}{c}{\!\!$m_1^2m_2$\!\!} &
        \multicolumn{1}{c}{}        \end{tabular}
\caption{A permutation table $\tau \in  F_{4,9}$ with $w(\tau) = m_1^{12}m_2^{7}m_3^2m_4$ and $\sgn(\tau) = -1$.}
\label{fig:exatablf}
\end{figure}
\end{example}

\begin{proposition}For any distribution $\Omega$, we have $    f_k(n) = \sum_{\tau \in F_{k,n}} w(\tau) \sgn(\tau)$.
\end{proposition}
\begin{proof}
This follows directly from the expansion $\det A =  \sum_{\pi \in F_n} \sgn(\pi) \prod_{i \in \left[n\right]} A_{i \pi(i)}$ raised to the $k$-th power and by taking expectation as $f_k(n) = \Exx \left[(\det A)^k\right]$.
\end{proof}

\begin{example}
Figure \ref{Fig:CorPT} shows how we can compute $f_k(n)$ using the tables in $F_k$ for $n = 2$ and $k = 2$. Summing up the contributions from these tables, we have that $f_2(2) = 2m_2^2-2m_1^4 = 2(m_2-m_1^2)(m_2+m_1^2)$.
\begin{figure}[tbh!]
\centering
\setlength{\tabcolsep}{3.2pt}
\vspace{-1em}
\begin{GrayBox}
\begin{tabular}{>{\centering\arraybackslash}m{12ex}
>{\centering\arraybackslash}m{12ex}
>{\centering\arraybackslash}m{16ex}
>{\centering\arraybackslash}m{16ex}
>{\centering\arraybackslash}m{16ex}}
$\hspace{2ex}(\det A)^2 \hspace{2ex}=$ & $A_{1\mathbf{1}}^2A_{2\mathbf{2}}^2$ & $\hspace{-3.5ex}-\hspace{1.5ex}A_{1\mathbf{1}}A_{2\mathbf{2}}A_{12}A_{21}$ & $\hspace{-1ex}- \hspace{1ex}A_{12}A_{21}A_{1\mathbf{1}}A_{2\mathbf{2}}$ & $\hspace{-4.5ex}+\hspace{2ex} A_{12}^2A_{21}^2$\\[0.45em]
    $F_{2,2} :$ & \begin{tabular}{|>{\columncolor{mycolor}}c|>{\columncolor{mycolor}}c|}
    \hline
        $\mathbf{1}$ & $\mathbf{2}$\\
        $\mathbf{1}$ & $\mathbf{2}$\\
    \hline
    \end{tabular}
    &
        \begin{tabular}{|>{\columncolor{mycolor}}c|>{\columncolor{mycolor}}c|}
    \hline
        $\mathbf{1}$ & $\mathbf{2}$\\
        $2$ & $1$\\
    \hline
    \end{tabular}
    &
        \begin{tabular}{|>{\columncolor{mycolor}}c|>{\columncolor{mycolor}}c|}
    \hline
        $2$ & $1$\\
        $\mathbf{1}$ & $\mathbf{2}$\\
    \hline
    \end{tabular}
    &
        \begin{tabular}{|>{\columncolor{mycolor}}c|>{\columncolor{mycolor}}c|}
    \hline
        $2$ & $1$\\
        $2$ & $1$\\
    \hline
    \end{tabular}
    \\[1.4em]
    Weight: & $m_2m_2$ & $m_1^2m_1^2$ & $m_1^2 m_1^2$ & $m_2m_2$\\[0.05em]
    Sign: & $+$ & $-$ & $-$ & $+$
\end{tabular}
\end{GrayBox}
\vspace{-1em}
\caption{Correspondence between 
$f_2(2)$ and the permutation tables in $F_{2,2}$}
\label{Fig:CorPT}
\end{figure}
\end{example}

\begin{definition}
\label{def:EGF}
We let $\mathcal{F}_{k,n}$ denote the tables in $F_{k,n}$ where the column order is irrelevant (i.e., two tables in $\mathcal{F}_{k,n}$ are equivalent if one table can be obtained from the other by permuting its columns) and we define $\mathcal{F}_k = \cup_{n=0}^\infty \mathcal{F}_{k,n}$. 
These tables now have exponential generating functions and can thus be analysed by tools of analytic combinatorics for labeled combinatorial structures. Also, we define 
$\hat{f}_k(n) = \sum_{t \in \mathcal{F}_{k,n}} w(t) \sgn t$. Since there are $n!$ ways how we can permute the columns, we have that $f_k(n)=n! \hat{f}_k(n)$.
\end{definition}

\begin{example}
Let us compute $f_2(3)$. We may write $f_2(3) = 3! \hat{f}_2(3)$, where $\hat{f}_2(3)= \sum_{t\in \mathcal{F}_{2,3}} w(t)\sgn t$. Figure \ref{Fig:CorPT2} lists all elements of $\mathcal{F}_{2,3}$ and shows their weights and signs. Summing the contributions, we get $\hat{f}_2(3) = m_2^3-3m_2m_1^4+2m_1^6 = (m_2+2m_1^2)(m_2-m_1^2)^2$ and thus $f_2(3) = 3!(m_2+2m_1^2)(m_2-m_1^2)^2$.
\begin{figure}[tbh!]
\centering
\setlength{\tabcolsep}{3.2pt}
\vspace{-1em}
\begin{GrayBox}
\begin{tabular}{>{\centering\arraybackslash}m{7ex}
>{\centering\arraybackslash}m{10ex}
>{\centering\arraybackslash}m{10ex}
>{\centering\arraybackslash}m{10ex}
>{\centering\arraybackslash}m{10ex}
>{\centering\arraybackslash}m{10ex}
>{\centering\arraybackslash}m{10ex}}
    $\mathcal{F}_{2,3} :$ &
    \begin{tabular}{|>{\columncolor{mycolor}}c|>{\columncolor{mycolor}}c|>{\columncolor{mycolor}}c|}
    \hline
        $1$ & $2$ & $3$\\
        $1$ & $2$ & $3$\\
    \hline
    \end{tabular}
    &
    \begin{tabular}{|>{\columncolor{mycolor}}c|>{\columncolor{mycolor}}c|>{\columncolor{mycolor}}c|}
    \hline
        $1$ & $2$ & $3$\\
        $1$ & $3$ & $2$\\
    \hline
    \end{tabular}
    &
    \begin{tabular}{|>{\columncolor{mycolor}}c|>{\columncolor{mycolor}}c|>{\columncolor{mycolor}}c|}
    \hline
        $1$ & $2$ & $3$\\
        $3$ & $2$ & $1$\\
    \hline
    \end{tabular}
    &
    \begin{tabular}{|>{\columncolor{mycolor}}c|>{\columncolor{mycolor}}c|>{\columncolor{mycolor}}c|}
    \hline
        $1$ & $2$ & $3$\\
        $2$ & $1$ & $3$\\
    \hline
    \end{tabular}
        &
    \begin{tabular}{|>{\columncolor{mycolor}}c|>{\columncolor{mycolor}}c|>{\columncolor{mycolor}}c|}
    \hline
        $1$ & $2$ & $3$\\
        $3$ & $1$ & $2$\\
    \hline
    \end{tabular}
    &
    \begin{tabular}{|>{\columncolor{mycolor}}c|>{\columncolor{mycolor}}c|>{\columncolor{mycolor}}c|}
    \hline
        $1$ & $2$ & $3$\\
        $2$ & $3$ & $1$\\
    \hline
    \end{tabular}
    \\[1.4em]
    Weight: & $m_2m_2m_2$ & $m_2m_1^2m_1^2$ & $m_1^2 m_2m_1^2$ & $m_1^2m_1^2m_2$ & $m_1^2m_1^2m_1^2$ & $m_1^2m_1^2m_1^2$\\[0.05em]
    Sign: & $+$ & $-$ & $-$ & $-$ & $+$ & $+$
\end{tabular}
\end{GrayBox}
\vspace{-1em}
\caption{Correspondence between $\hat{f}_2(3)$ and the permutation tables in $\mathcal{F}_{2,3}$}
\label{Fig:CorPT2}
\end{figure}
\end{example}

\subsection{Marked permutations and tables}

Instead of expressing $f_k(n)$ in terms of $m_r$, we can express it in terms of $\mu_r = \Exx [B_{ij}^r]$ where $B_{ij} = A_{ij} - m_1$.


\begin{definition}
We say $\sigma$ is a \emph{marked permutation} if it is either in $S_n$ (i.e., $\sigma$ is a permutation of $[n]$) or $\sigma$ is formed from some $\pi \in S_n$ by replacing at most one element by the mark ``$\times$''. 
We define $\sgn(\sigma) = \sgn(\pi)$ and $B^\times_{i\sigma(i)} = m_1$ if $i$ is marked and $B^\times_{i\sigma(i)} = B_{i\pi(i)}$ otherwise. We write $S^\times_n$ for the set of all marked permutations.
\end{definition}
\begin{proposition} In terms of marked permutations, $\det(A) = \sum_{\sigma \in S^\times_n} \sgn(\sigma) \prod_{i=1}^n B_{i \sigma(i)}^\times$.
\end{proposition}

\begin{definition}
We say $\tau$ is a marked $k$-table with $n$ columns if its rows are marked permutations $\sigma_j, j = 1,\ldots,k$ of order $n$. We define $G^{\times}_{k,n}=S_n^\times\times\cdots \times S_n^\times$ ($k$ copies) to be the set of all such tables and let $G^\times_k=\sum_{n=0}^\infty G_{k,n}^\times$ be the set of all marked tables with $k$ rows. We define the marked weight $w$ of the $i$-th column of $\tau \in G_{k,n}^{\times}$ as the expectation $\Exx \left[\prod_{j=1}^k B^\times_{i\sigma_j(i)}\right]$. Similarly, we define the sign $\sgn(\tau)$ of $\tau$ to be the product of the signs of $\sigma_j, j=1,\ldots,k$ and we define the marked weight $w(\tau)$ of $\tau$ to be the product of the weights of its individual columns.
\end{definition}

\begin{example}
Figures \ref{fig:extratable0} and \ref{fig:extratable} show two examples of marked permutation tables.
\renewcommand{\arraystretch}{0.9}
\begin{figure}[tbh!]
\centering
\begin{minipage}{.49\textwidth}
        \centering
    \begin{tabular}{|>{\columncolor{mycolor}}c|>{\columncolor{mycolor}}c|>{\columncolor{mycolor}}c|>{\columncolor{mycolor}}c|>{\columncolor{mycolor}}c|>{\columncolor{mycolor}}c|>{\columncolor{mycolor}}c|>{\columncolor{mycolor}}c|>{\columncolor{mycolor}}c|}
    \hhline{|=========|}
        \ru{0.9}1 & $\times$ & 3 & 4 & 5 & 2 & 7 & 8 & 9 \\
        3 & 2 & 1 & 9 & 4 & 6 & 7 & 5 & 8 \\
        1 & $\times$ & 3 & 9 & 4 & 2 & 7 & 5 & 8 \\ 
        3 & 2 & 1 & 4 & 5 & 6 & 7 & 8 & 9 \\
    \hhline{|=========|}
    \end{tabular}
\caption{$\tau \in  G_{4,9}^{2}$, $w(\tau) = m_1^2\mu_2^{15}\mu_4$.}
\label{fig:extratable0}
    \end{minipage}%
\hfill
\begin{minipage}{.49\textwidth}
        \centering
    \begin{tabular}{|>{\columncolor{mycolor}}c|>{\columncolor{mycolor}}c|>{\columncolor{mycolor}}c|>{\columncolor{mycolor}}c|>{\columncolor{mycolor}}c|>{\columncolor{mycolor}}c|>{\columncolor{mycolor}}c|>{\columncolor{mycolor}}c|>{\columncolor{mycolor}}c|}
    \hhline{|=========|}
        \ru{0.9}$\times$ & 2 & 3 & 4 & 5 & 6 & 7 & 8 & 9 \\
        $\times$ & 2 & 1 & 9 & 4 & 6 & 7 & 5 & 8 \\ 
        2 & $\times$ & 1 & 9 & 4 & 6 & 7 & 5 & 8 \\ 
        2 & $\times$ & 3 & 4 & 5 & 6 & 7 & 8 & 9 \\
    \hhline{|=========|}
    \end{tabular}
\caption{$\tau \in  G_{4,9}^{4}$, $w(\tau) = m_1^4\mu_2^{12}\mu_4^2$.}
\label{fig:extratable}
    \end{minipage}%
\end{figure}
\vspace{-1em}
\end{example}

\noindent
Since $\mu_1 =0$, it turns out that most tables in $G^\times_{k,n}$ have $w(\tau) = 0$.
\begin{definition}
We say a table $\tau \in G^\times_{k,n}$ is trivial if its weight vanishes, otherwise the table is nontrivial. The set all nontrivial tables form a subset $T^\times_{k,n} \subseteq G^\times_{k,n}$.
\end{definition}
\begin{proposition}\label{prop:binomS}
For any distribution $\Omega$, we have $f_k(n) = \sum_{\tau \in T_{k,n}^{\times}} w (t) \sgn(\tau)$.
\end{proposition}

\begin{example}
The correspondence between $f_k(n)$ and marked permutation tables is shown below 
for $n=2$ and $k=2$. By summing up the contributions from all nontrivial marked tables, since $\mu_2=m_2-m_1^2$, we again get
\begin{equation*}
f_2(2) = 2\mu_2^2+4m_1^2\mu_2 = 2(m_2-m_1^2)(m_2+m_1^2).
\end{equation*}

\begin{figure}[H]
\centering
\setlength{\tabcolsep}{3.2pt}
\vspace{-1em}
\begin{GrayBox}
\vspace{-0.6em}
\hspace{-0.5em}
\begin{tabular}{>{\centering\arraybackslash}m{10ex}
>{\centering\arraybackslash}m{10ex}
>{\centering\arraybackslash}m{10ex}
>{\centering\arraybackslash}m{10ex}
>{\centering\arraybackslash}m{10ex}
>{\centering\arraybackslash}m{10ex}
>{\centering\arraybackslash}m{16ex}}
$\hspace{1.2ex}(\det A)^2 \hspace{0.5ex}=$ & $B_{11}^2B_{22}^2$ & $\hspace{-2.8ex} +\hspace{1.3ex} B_{12}^2B_{21}^2$ & $\hspace{-2.8ex} + \hspace{1.3ex} m_1^2B_{22}^2$ & $\hspace{-2.8ex} +\hspace{1.3ex} m_1^2B_{21}^2$ & $\hspace{-3.1ex} +\hspace{1.3ex} B_{11}^2m_1^2$ & $\hspace{-9.5ex} +\hspace{1ex} B_{12}^2m_1^2$\\[0.45em]
    $T^{\times}_{2,2} :$ & \begin{tabular}{|>{\columncolor{mycolor}}c|>{\columncolor{mycolor}}c|}
    \hline
        $1$ & $2$\\
        $1$ & $2$\\
    \hline
    \end{tabular}
    &
        \begin{tabular}{|>{\columncolor{mycolor}}c|>{\columncolor{mycolor}}c|}
    \hline
        $2$ & $1$\\
        $2$ & $1$\\
    \hline
    \end{tabular}
    &
        \begin{tabular}{|>{\columncolor{mycolor}\centering\arraybackslash}m{1ex}|>{\columncolor{mycolor}}c|}
    \hline
        $\!\times$ & $2$\\
        $\!\times$ & $2$\\
    \hline
    \end{tabular}
    &
        \begin{tabular}{|>{\columncolor{mycolor}\centering\arraybackslash}m{1ex}|>{\columncolor{mycolor}}c|}
    \hline
        $\!\times$ & $1$\\
        $\!\times$ & $1$\\
    \hline
    \end{tabular}
    &
        \begin{tabular}{|>{\columncolor{mycolor}}c|>{\columncolor{mycolor}\centering\arraybackslash}m{1ex}|}
    \hline
        $1$ & $\!\times$\\
        $1$ & $\!\times$\\
    \hline
    \end{tabular}
    &
        \hspace{-3em}\begin{tabular}{|>{\columncolor{mycolor}}c|>{\columncolor{mycolor}\centering\arraybackslash}m{1ex}|}
    \hline
        $2$ & $\!\times$\\
        $2$ & $\!\times$\\
    \hline
    \end{tabular}
    \\[1.4em]
    $w$: & $\mu_2\mu_2$ & $\mu_2\mu_2$ & $m_1^2 \mu_2$ & $m_1^2 \mu_2$ & $\mu_2 m_1^2$ & $\hspace{-3em}\mu_2 m_1^2$\\[0.0em]
    Sign: & $+$ & $+$ & $+$ & $+$ & $+$ & $\hspace{-3em}+$\\
\end{tabular}
\vspace{-0.8em}
\end{GrayBox}
\vspace{-1em}
\caption{Correspondence between $f_2(2)$ 
and marked permutation tables}
\label{Fig:CorPTM}
\end{figure}
\vspace{-1em}
\end{example}

\begin{definition}
\label{def:EGFmarked}
Similarly as in Definition \ref{def:EGF}, we let 
$\mathcal{G}^\times_{k,n}$ and $\mathcal{T}^\times_{k,n}$ 
denote the sets of tables $G^\times_{k,n}$ and $T^\times_{k,n}$ where the column order is irrelevant. 
\end{definition}

\begin{example}\label{Ex:T42x}
Let us derive $f_4(2)$. We may write $f_4(2) = 2! \hat{f}_{\!4}(2)$ and sum the contributions from tables in $\mathcal{T}_{4,2}^\times$. 
Figure \ref{Fig:CorPTMIrrel} below shows the members of $\mathcal{T}_{4,2}^\times$ with the corresponding sign and weight including multiplicity $(\#)$ as the members are displayed up to 
permutation of rows and substitution of $\{1,2\}$ for the elements $\{a,b\}$.

\begin{figure}[H]
\centering
\setlength{\tabcolsep}{3.2pt}
\vspace{-1em}
\begin{GrayBox}
\hspace{-1em}
\begin{tabular}{>{\centering\arraybackslash}m{4.5ex}
>{\centering\arraybackslash}m{4.5ex}
>{\centering\arraybackslash}m{4.5ex}
>{\centering\arraybackslash}m{6.5ex}
>{\centering\arraybackslash}m{6.5ex}
>{\centering\arraybackslash}m{5ex}
>{\centering\arraybackslash}m{5ex}
>{\centering\arraybackslash}m{6.5ex}
>{\centering\arraybackslash}m{5ex}
>{\centering\arraybackslash}m{5ex}
>{\centering\arraybackslash}m{5ex}
>{\centering\arraybackslash}m{5ex}}
    \begin{tabular}{c}
           \\
         $\!\mathcal{T}^{\times}_{4,2} :$ 
    \end{tabular} & \begin{tabular}{|>{\columncolor{mycolor}}c|>{\columncolor{mycolor}}c|}
    \hline
        $a$ & $b$\\
        $a$ & $b$\\
        $a$ & $b$\\
        $a$ & $b$\\
    \hline
    \end{tabular}
    &
        \begin{tabular}{|>{\columncolor{mycolor}}c|>{\columncolor{mycolor}}c|}
    \hline
        $a$ & $b$\\
        $a$ & $b$\\
        $b$ & $a$\\
        $b$ & $a$\\
    \hline
    \end{tabular}
    &
        \begin{tabular}{|>{\columncolor{mycolor}\centering\arraybackslash}m{1ex}|>{\columncolor{mycolor}}c|}
    \hline
        $\!\times$ & $b$\\
        $a$ & $b$\\
        $a$ & $b$\\
        $a$ & $b$\\
    \hline
    \end{tabular}
    &
        \begin{tabular}{|>{\columncolor{mycolor}\centering\arraybackslash}m{1ex}|>{\columncolor{mycolor}}c|}
    \hline
        $\!\times$ & $b$\\
        $\!\times$ & $b$\\
        $a$ & $b$\\
        $a$ & $b$\\
    \hline
    \end{tabular}
    &
        \begin{tabular}{|>{\columncolor{mycolor}\centering\arraybackslash}m{1ex}|>{\columncolor{mycolor}}c|}
    \hline
        $\!\times$ & $b$\\
        $\!\times$ & $b$\\
        $b$ & $a$\\
        $b$ & $a$\\
    \hline
    \end{tabular}
    &
        \begin{tabular}{|>{\columncolor{mycolor}\centering\arraybackslash}m{1ex}|>{\columncolor{mycolor}}c|}
    \hline
        $\!\times$ & $b$\\
        $a$ & $\!\times$\\
        $a$ & $b$\\
        $a$ & $b$\\
    \hline
    \end{tabular}
    &
    \begin{tabular}{|>{\columncolor{mycolor}\centering\arraybackslash}m{1ex}|>{\columncolor{mycolor}\centering\arraybackslash}m{1ex}|}
    \hline
        $\!\times$ & $b$\\
        $a$ & $\!\times$\\
        $a$ & $\!\times$\\
        $a$ & $b$\\
    \hline
    \end{tabular}
    &
        \begin{tabular}{|>{\columncolor{mycolor}\centering\arraybackslash}m{1ex}|>{\columncolor{mycolor}}c|}
    \hline
        $\!\times$ & $b$\\
        $\!\times$ & $b$\\
        $\!\times$ & $b$\\
        $\!\times$ & $b$\\
    \hline
    \end{tabular}
    &
        \begin{tabular}{|>{\columncolor{mycolor}\centering\arraybackslash}m{1ex}|>{\columncolor{mycolor}}c|}
    \hline
        $\!\times$ & $b$\\
        $\!\times$ & $b$\\
        $\!\times$ & $a$\\
        $\!\times$ & $a$\\
    \hline
    \end{tabular}
    &
    \begin{tabular}{|>{\columncolor{mycolor}\centering\arraybackslash}m{1ex}|>{\columncolor{mycolor}\centering\arraybackslash}m{1ex}|}
    \hline
        $\!\times$ & $b$\\
        $\!\times$ & $b$\\
        $b$ & $\!\times$\\
        $b$ & $\!\times$\\
    \hline
    \end{tabular}
    &
    \begin{tabular}{|>{\columncolor{mycolor}\centering\arraybackslash}m{1ex}|>{\columncolor{mycolor}\centering\arraybackslash}m{1ex}|}
    \hline
        $\!\times$ & $b$\\
        $\!\times$ & $b$\\
        $a$ & $\!\times$\\
        $a$ & $\!\times$\\
    \hline
    \end{tabular}
    \\[3.2em]
    $w_\times$: & $\mu_4^2$ & $\mu_2^4$ & $\!\! m_1\mu_3\mu_4$ & $\!\!m_1^2 \mu_2\mu_4$ & $m_1^2 \mu_2^3$ & $m_1^2\mu_3^2$ & $\!\! m_1^3\mu_2\mu_3$ & $m_1^4\mu_4$ & $m_1^4\mu_2^2$ & $m_1^4\mu_2^2$ & $m_1^4\mu_2^2$\\[0.1em]
    Sign: & $+$ & $+$ & $+$ & $+$ & $+$ & $+$ & $+$ & $+$ & $+$ & $+$ & $+$\\[0.1em]
    $\#$ & $1$ & $3$ & $8$ & $12$ & $12$ & $12$ & $24$ & $2$ & $6$ & $6$ & $6$\\[0.1em]
\end{tabular}
\vspace{-0.4em}
\end{GrayBox}
\vspace{-1em}
\caption{Correspondence between 
$f_4(2)$ and marked permutation tables $\mathcal{T}_{4,2}^\times$}
\label{Fig:CorPTMIrrel}
\end{figure}

By summing the contribution from all marked tables, we get $\hat{f}_{\!4}(2) = \mu _4^2+3 \mu _2^4+8 m_1 \mu _3 \mu _4+12 m_1^2 \mu _2 \mu _4+12 m_1^2 \mu _2^3 +12 m_1^2\mu _3^2+24 m_1^3 \mu _2 \mu _3+2 m_1^4 \mu _4+18 m_1^4 \mu _2^2$. Plugging $\mu_2 = m_2-m_1^2$, $\mu _3 = m_3 -3 m_2 m_1 + 2 m_1^3$, $\mu _4= m_4-4 m_3 m_1 +6 m_2m_1^2-3 m_1^4$ and expanding, we get $f_4(2) = 2! \hat{f}_{\!4}(2) = 2(m_4^2-4 m_1^2 m_3^2+3 m_2^4)$, which coincides with the introductory Example \ref{Ex:f42}.
\end{example}

\section{General fourth moment
}
In this section, we generalise the result for $F_4(t)|_{m_1=0}$ by Nyquist, Rice and Riordan.
\begin{theorem}[Beck 2023 \cite{beck2022fourth}]\label{Thm:FourthMoment}
For any distribution of $\Omega$,
\begin{equation*}
F_4(t) = \frac{e^{t(\mu_4-3 \mu _2^2)}}{(1-\mu_2^2 t)^3} \left((1+m_1\mu_3 t)^4+6m_1^2\mu_2 t\frac{(1+m_1 \mu_3t)^2}{1-\mu _2^2 t}+m_1^4 t\frac{1+7 \mu_2^2 t+4 \mu _2^4 t^2}{(1-\mu _2^2 t)^2}\right).
\end{equation*}
\end{theorem}
\begin{corollary}
Furthermore, defining $d_0(c) = 2+c$, $d_1(c) = c(2+c)$ and $d_2(c) = c^3$,
\begin{equation*}
f_4(n) = n!^2 \sum _{w=0}^2 \sum _{s=0}^{4-2w} \sum _{c=0}^{n-s} \binom{4-2 w}{s}\frac{(1+c) m_1^{s+2 w} \mu_2^{2c-w} \mu _3^s \left(\mu _4-3\mu_2^2\right){}^{n-c-s}}{(n-c-s)!(2-w)! w!} d_w(c).
\end{equation*}
\end{corollary}

\begin{proof}
Without the loss of generality, we set $\mu_2 = 1$ (the general case is obtained by the scaling property of determinants). Let $a,b$ denote different numbers selected from $[n]=\{1,2,3,\ldots,n\}$. Up to a permutation of rows, the only ways how the columns of $4$ by $n$ tables with nonzero weight could look like are as follows:
\vspace{-1em}
\begin{figure}[H]
\centering
\setlength{\tabcolsep}{5.2pt}
\begin{GrayBox}
\vspace{-0.6em}
\centering
\setlength{\tabcolsep}{5.2pt} \begin{tabular}{cccccc}
    Type: & $4$-column & $2$-column & $\times^1$-column & $\times^2$-column & $\times^4$-column \\[0.5em]
    $\mathcal{T}_4^\times :$ & \begin{tabular}{|>{\columncolor{mycolor}}c|}
    \hline
        $a$\\
        $a$\\
        $a$\\
        $a$\\
    \hline
    \end{tabular}
    &
        \begin{tabular}{|>{\columncolor{mycolor}}c|}
    \hline
        $a$\\
        $a$\\
        $b$\\
        $b$\\
    \hline
    \end{tabular}
    &
        \begin{tabular}{|>{\columncolor{mycolor}}c|}
    \hline
        $\times$\\
        $a$\\
        $a$\\
        $a$\\
    \hline
    \end{tabular}
    &
        \begin{tabular}{|>{\columncolor{mycolor}}c|}
    \hline
        $\times$\\
        $\times$\\
        $a$\\
        $a$\\
    \hline
    \end{tabular}
    &
        \begin{tabular}{|>{\columncolor{mycolor}}c|}
    \hline
        $\times$\\
        $\times$\\
        $\times$\\
        $\times$\\
    \hline
    \end{tabular}
    \\[1.9em]
    Weight $w_\times$: & $\mu_4$ & $1$ & $m_1\mu_3$ & $m_1^2$ & $m_1^4$
\end{tabular}
\vspace{-0.5em}
\end{GrayBox}
\vspace{-1em}
\caption{Structure of $\mathcal{T}^\times_4$ tables}
\end{figure}    
\vspace{-1em}

\noindent
The $\times^1$ columns contain a single element $a$, one instance of which is covered by $\times$, hence they are \textbf{disjoint} from the rest of a table. As a result, we can consider only tables $\mathcal{S}_4^\times \subset \mathcal{T}_4^\times$ which do not contain $\times^1$-columns. In any given table $\tau \in \mathcal{T}_4^\times$, there could be either four, two or no $\times^1$-columns. In terms of generating functions, this corresponds to
\begin{equation}\label{eq:F4expan}
F_4(t) = (1+ m_1\mu_3 t)^4S^0_4(t) + (1+ m_1\mu_3 t)^2S^2_4(t) + S^4_4(t),
\end{equation}
where $S^r_4(t)$ denotes the EGF of tables $\mathcal{S}^r_4 \subset \mathcal{S}^\times_4 $ with $r$ marks containing no $\times^1$ columns. It is convenient to let $\mathcal{S}^{r/s}_4\subseteq \mathcal{S}^r_4$ denote tables whose marks are distributed among exactly $s$ different columns. Since there is at most one mark per row, $\mathcal{S}^{r/s}_4$ tables contain only a few marked columns (see below). Structurally, $\mathcal{S}_4^\times = \mathcal{S}_4^0 + \mathcal{S}_4^2 + \mathcal{S}_4^4$ where $\mathcal{S}_4^4 = \mathcal{S}_4^{4/1} + \mathcal{S}_4^{4/2}$.
\vspace{-1.5em}
\begin{figure}[H]
\centering
\setlength{\tabcolsep}{5.2pt}
\begin{GrayBox}
\centering
\setlength{\tabcolsep}{5.2pt} \begin{tabular}{>{\centering\arraybackslash}m{8ex}
>{\centering\arraybackslash}m{12ex}
>{\centering\arraybackslash}m{12ex}
>{\centering\arraybackslash}m{12ex}
>{\centering\arraybackslash}m{12ex}}
    \begin{tabular}{c}
    \\[0.5em]
    $\mathcal{S}_4^\times :$ 
    \end{tabular} &
    \begin{tabular}{|>{\columncolor{mycolor}}c|}
    \hline
        \hspace{0.75em} \\
        \\
        \\
        \\
    \hline
    \end{tabular}
    &
    \begin{tabular}{|>{\columncolor{mycolor}}c|}
    \hline
        $\times$ \\
        $\times$ \\
        \\
        \\
    \hline
    \end{tabular}
    &
    \begin{tabular}{|>{\columncolor{mycolor}}c|}
    \hline
        $\times$ \\
        $\times$ \\
        $\times$ \\
        $\times$ \\
    \hline
    \end{tabular}
    &
    \begin{tabular}{|>{\columncolor{mycolor}}c|>{\columncolor{mycolor}}c|}
    \hline
        $\times$ & \\
        $\times$ & \\
        & $\times$\\
        & $\times$\\
    \hline
    \end{tabular}
    \\[3.4em]
    & $\mathcal{S}^0_4$ & $\mathcal{S}^2_4$ & $\mathcal{S}^{4/1}_4$ & $\mathcal{S}^{4/2}_4$
\end{tabular}
\end{GrayBox}
\vspace{-1em}
\caption{Structure of $\mathcal{S}^\times_4$ tables}
\label{fig:Sxtables}
\end{figure}    
\vspace{-1.5em}

\begin{proposition}[Nyquist, Rice and Riordan 1954 \cite{nyquist1954distribution}]\label{Prop:NRR}
\begin{equation*}
S_4^0(t) = F_4(t)|_{m_1=0,m_4 \to \mu_4} = \frac{e^{t(\mu_4-3)}}{(1-t)^3}
\end{equation*}
\end{proposition}
\begin{proof}
Since $\mu_4$ equals $m_4$ when $m_1 = 0$, $S_4^0(t)$ coincides with the expression for $F_4(t)|_{m_1=0}$ obtained by Nyquist, Rice and Riordan. In $\mathcal{S}_4^0$ tables, the $4$-columns are disjoint from the remaining $2$-columns. Furthermore, the 2-columns can be further divided into disjoint components. To a given table of 2-columns, we can associate a derangement $\pi$ whose cycles correspond to disjoint sub-tables into which this table decomposes. To make this association a bijection, there are $3$ ways how the remaining elements in a given sub-table can be arranged (see Figure \ref{tDtable}, each arrangement is represented by a vertical box with four slots filled with two dots representing in which rows the number in the first row appears). Since each row appears twice in any sub-table, the overall sign of those sub-tables is always positive.
\renewcommand{\arraystretch}{0.95}
\begin{figure}[H]
\centering
\begin{minipage}{.5\textwidth}
        \centering
    \begin{tabular}{|>{\columncolor{mycolor}}c|>{\columncolor{mycolor}}c|>{\columncolor{mycolor}}c|>{\columncolor{mycolor}}c|>{\columncolor{mycolor}}c|>{\columncolor{mycolor}}c|>{\columncolor{mycolor}}c|>{\columncolor{mycolor}}c|>{\columncolor{mycolor}}c|}
    \hhline{|==}
        1 & 3 \\ \hline
        1 & 3\\ 
        3 & 1\\ 
        3 & 1\\
    \hhline{|==}
    \end{tabular}
    \begin{tabular}{|>{\columncolor{mycolor}}c|>{\columncolor{mycolor}}c|>{\columncolor{mycolor}}c|}
    \hhline{===}
        2 & 6 & 7\\ \hline
        6 & 7 & 2\\ 
        2 & 6 & 7 \\ 
        6 & 7 & 2\\
    \hhline{===}
    \end{tabular}
    \begin{tabular}{|>{\columncolor{mycolor}}c|>{\columncolor{mycolor}}c|>{\columncolor{mycolor}}c|>{\columncolor{mycolor}}c|}
    \hhline{====|}
        4 & 5 & 8 & 9 \\ \hline
        9 & 4 & 5 & 8 \\ 
        9 & 4 & 5 & 8 \\ 
        4 & 5 & 8 & 9 \\
    \hhline{====|}
    \end{tabular}
    \end{minipage}%
\begin{minipage}{.5\textwidth}
        \centering
\begin{tikzpicture}[thick, main/.style = {draw,circle, inner sep = 2pt}, scale = 0.8]
\node[main] (1) at (-2.5,1) {$1$};
\node[main] (3) at (-2.5,-1) {$3$};
\node[main] (2) at (-1,0.8) {$2$};
\node[main] (6) at (1,0.8) {$6$}; 
\node[main] (7) at (0,-1) {$7$};
\node[main] (4) at (2,0) {$4$};
\node[main] (9) at (3,1) {$9$};
\node[main] (8) at (4,0) {$8$};
\node[main] (5) at (3,-1) {$5$};
\draw [fill=mycolor, draw=black] (-2.6,-0.4) rectangle (-2.4,0.4);
\draw[black,fill=black] (-2.5,0.3) circle (0.05);
\draw[black,fill=black] (-2.5,0.1) circle (0.05);
\draw [fill=mycolor, draw=black] (-0.1,-0.2) rectangle (0.1,0.6);
\draw[black,fill=black] (0,0.5) circle (0.05);
\draw[black,fill=black] (0,0.1) circle (0.05);
\draw [fill=mycolor, draw=black] (2.9,-0.4) rectangle (3.1,0.4);
\draw[black,fill=black] (3,0.3) circle (0.05);
\draw[black,fill=black] (3,-0.3) circle (0.05);
\draw[->] (1) to [bend left=40] (3);
\draw[->] (3) to [bend left=40] (1);
\draw[->] (2) to [bend left=40] (6);
\draw[->] (6) to [bend left=40] (7);
\draw[->] (7) to [bend left=40] (2);
\draw[->] (9) to [bend left=20] (8);
\draw[->] (8) to [bend left=20] (5);
\draw[->] (5) to [bend left=20] (4);
\draw[->] (4) to [bend left=20] (9);
\end{tikzpicture}
    \end{minipage}%
\caption{One-to-one correspondence between a table $\tau$ with nine $2$-columns decomposable into three disjoint sub-tables, and its associated derangement $\pi$ with cycles labeled according to the repetitions of the number in the first row of $\tau$.}
\label{tDtable}
\end{figure}
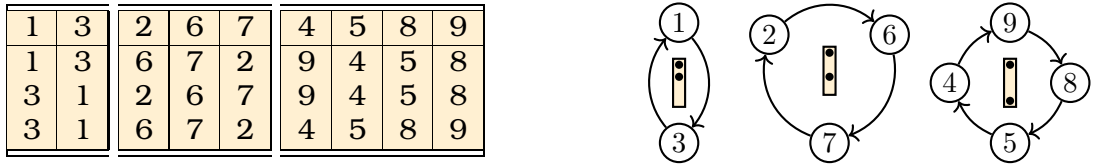

\noindent
Hence, structurally, $\mathcal{S}_4^0 = \textsc{Set}\big{(}\, m_4 \,\, \oneincirc \,\big{)} \star\, \textsc{Set}\big{(}+3\,\,\textsc{Cyc}_{\geq 2}\big{(} +\, \oneincirc \,\big{)}\big{)}$ and thus in terms of generating functions, $S_4^0(t) = \exp(\mu_4 t) \exp(-3t - \ln(1-t)) = \frac{e^{t(\mu_4-3)}}{(1-t)^3}$.
\end{proof}

\begin{proposition}
\begin{equation*}
S^2_4(t) = m_1^2(6-2\mu_4)\frac{\partial S^{0}_4(t)}{\partial \mu_4}+2m_1^2t\,\frac{\partial S^{0}_4(t)}{\partial t} = \frac{6 m_1^2 t e^{t \left(\mu_4-3\right)}}{(1-t)^4}
\end{equation*}
\end{proposition}

\begin{proof}
Let $\tau' \in \mathcal{S}^0_{4,n}$ have $c$ $4$-columns and thus $n-c$ $2$-columns. Its weight is $\mu_4^c$. From this $\tau'$, we create $\tau \in \mathcal{S}^2_{4,n}$ by covering one pair of identical elements by marks in either a $4$-column or a $2$-column. The contribution of $\tau'$ to $\sum_{\tau \in \mathcal{S}^{2}_{4,n}} w(\tau)\sgn\tau$ is then $$6cm_1^2\mu_4^{c-1}+2(n-c)m_1^2\mu_4^{c} = m_1^2(6-2\mu_4) \frac{\partial \mu_4^c}{\partial \mu_4} + 2m_1^2 n\mu_4^c.$$

\vspace{-1.5em}
\renewcommand{\arraystretch}{0.85}
\begin{table}[H]
\centering
\setlength{\tabcolsep}{5.2pt}
\begin{GrayBox}
\vspace{-0.3em}
\centering
\setlength{\tabcolsep}{5.2pt} \begin{tabular}{>{\centering\arraybackslash}m{20ex}
>{\centering\arraybackslash}m{20ex}
>{\centering\arraybackslash}m{20ex}}
    \begin{tabular}{c}
        \\[0.5em]
        $\mathcal{S}^2_{4,n} \leftarrow \mathcal{S}^0_{4,n}:$\\
    \end{tabular}
    &
    \begin{tabular}{|>{\columncolor{mycolor}}c|}
    \hline
        $\times$\\
        $\times$\\
        $a$\\
        $a$\\
    \hline
    \end{tabular}
$\leftarrow$
    \begin{tabular}{|>{\columncolor{mycolor}}c|>{\columncolor{mycolor}}c|}
    \hline
        \,$a$\,\\
        $a$\\
        $a$\\
        $a$\\
    \hline
    \end{tabular}
    &
    \begin{tabular}{|>{\columncolor{mycolor}}c|}
    \hline
        $\times$ \\
        $\times$ \\
        $b$ \\
        $b$ \\
    \hline
    \end{tabular}
$\leftarrow$
    \begin{tabular}{|>{\columncolor{mycolor}}c|>{\columncolor{mycolor}}c|}
    \hline
        \,$a$\,\\
        $a$\\
        $b$\\
        $b$\\
    \hline
    \end{tabular}
    \\[3.0em]
     & $6$ ways & $2$ ways
\end{tabular}
\vspace{-0.7em}
\end{GrayBox}
\end{table}    
\vspace{-1.5em}
\end{proof}

\begin{proposition}\label{Prop:S4/1}
\begin{equation*}
S^{4/1}_{4}(t) = m_1^4(1-\mu_4)\frac{\partial S^{0}_4(t)}{\partial \mu_4}+m_1^4 t\,\frac{\partial S^{0}_4(t)}{\partial t} = \frac{m_1^4t(1+2t)}{(1-t)^4} e^{t \left(\mu_4-3\right)}.
\end{equation*}
\end{proposition}

\begin{proof}
Let $\tau' \in \mathcal{S}^0_{4,n}$ have $c$ 4-columns and thus $n-c$ 2-columns, $w(\tau') = \mu_4^c$. From this $\tau'$, we create $\tau \in \mathcal{S}^{4/1}_{4,n}$ by covering an entire 4-column or 2-column. The contribution of $\tau'$ to $\sum_{\tau \in \mathcal{S}^{4/1}_{4,n}} w(\tau)\sgn\tau$ is $$cm_1^4\mu_4^{c-1}+(n-c)m_1^4\mu_4^c = m_1^4(1-\mu_4) \frac{\partial \mu_4^c}{\partial \mu_4} + m_1^4 n\mu_4^c.$$

\vspace{-1.5em}
\renewcommand{\arraystretch}{0.85}
\begin{table}[H]
\centering
\setlength{\tabcolsep}{5.2pt}
\begin{GrayBox}
\vspace{-0.3em}
\centering
\setlength{\tabcolsep}{5.2pt} \begin{tabular}{>{\centering\arraybackslash}m{20ex}
>{\centering\arraybackslash}m{20ex}
>{\centering\arraybackslash}m{20ex}}
    \begin{tabular}{c}
        \\[0em]
        $\mathcal{S}^{4/1}_{4,n} \leftarrow \mathcal{S}^0_{4,n}:$\\
    \end{tabular}
    &
    \begin{tabular}{|>{\columncolor{mycolor}}c|}
    \hline
        $\times$\\
        $\times$\\
        $\times$\\
        $\times$\\
    \hline
    \end{tabular}
$\leftarrow$
    \begin{tabular}{|>{\columncolor{mycolor}}c|>{\columncolor{mycolor}}c|}
    \hline
        \,$a$\,\\
        $a$\\
        $a$\\
        $a$\\
    \hline
    \end{tabular}
    &
    \begin{tabular}{|>{\columncolor{mycolor}}c|}
    \hline
        $\times$ \\
        $\times$ \\
        $\times$ \\
        $\times$ \\
    \hline
    \end{tabular}
$\leftarrow$
    \begin{tabular}{|>{\columncolor{mycolor}}c|>{\columncolor{mycolor}}c|}
    \hline
        \,$a$\,\\
        $a$\\
        $b$\\
        $b$\\
    \hline
    \end{tabular}
\end{tabular}
\vspace{-0.3em}
\end{GrayBox}
\end{table}    
\vspace{-1.5em}
\end{proof}

\begin{proposition}
\begin{equation*}
S^{4/2}_{4}(t) = (3-\mu_4)\frac{\partial S^{4/1}_4(t)}{\partial \mu_4}+t\,\frac{\partial S^{4/1}_4(t)}{\partial t} - S^{4/1}_4(t) = \frac{6m_1^4t^2(1+t)}{(1-t)^5} e^{t \left(\mu_4-3\right)}.
\end{equation*} 
\end{proposition}

\begin{proof}
Let $\tau' \in \mathcal{S}^{4/1}_{4,n}$ have $c$ 4-columns and thus $n-c-1$ 2-columns as now one column is a $\times^4$-column. The weight of $\tau'$ is $m_1^4\mu_4^c$. From this $\tau'$, we create $\tau \in \mathcal{S}^{4/2}_{4,n}$ by \textbf{swapping} its two $\times$ marks with a pair of numbers in a 4-column or 2-column. By symmetry, each table in $\mathcal{S}^{4/2}_{4,n}$ is counted twice. The contribution of $\tau'$ to $2\sum_{\tau \in \mathcal{S}^{4/2}_{4,n}} w(\tau)\sgn(\tau)$ is
\begin{equation*}
6cm_1^4\mu_4^{c-1}+2(n-c-1)m_1^4\mu_4^c = (6-2\mu_4) \frac{\partial (m_1^4\mu_4^c)}{\partial \mu_4} + 2n m_1^4\mu_4^c - 2m_1^4\mu_4^c
\end{equation*}

\vspace{-1.5em}
\renewcommand{\arraystretch}{0.85}
\begin{table}[H]
\centering
\setlength{\tabcolsep}{5.2pt}
\begin{GrayBox}
\vspace{-0.3em}
\centering
\setlength{\tabcolsep}{5.2pt} \begin{tabular}{>{\centering\arraybackslash}m{20ex}
>{\centering\arraybackslash}m{25ex}
>{\centering\arraybackslash}m{25ex}}
    \begin{tabular}{c}
        \\[0.5em]
        $2\mathcal{S}^{4/2}_{4,n} \leftarrow \mathcal{S}^{4/1}_{4,n}:$\\
    \end{tabular}
    &
   \begin{tabular}{|>{\columncolor{mycolor}}c|>{\columncolor{mycolor}}c|}
    \hline
        $\times$ & $a$\\
        $\times$ & $a$\\
        $a$ & $\times$\\
        $a$ & $\times$\\
    \hline
    \end{tabular}
$\leftarrow$
    \begin{tabular}{|>{\columncolor{mycolor}}c|>{\columncolor{mycolor}}c|}
    \hline
        $\times$ & \,$a$\,\\
        $\times$ & $a$\\
        $\times$ & $a$\\
        $\times$ & $a$\\
    \hline
    \end{tabular}
    &
    \begin{tabular}{|>{\columncolor{mycolor}}c|>{\columncolor{mycolor}}c|}
    \hline
        $\times$ & $a$\\
        $\times$ & $a$\\
        $b$ & $\times$\\
        $b$ & $\times$\\
    \hline
    \end{tabular}
$\leftarrow$
    \begin{tabular}{|>{\columncolor{mycolor}}c|>{\columncolor{mycolor}}c|}
    \hline
        $\times$ & \,$a$\,\\
        $\times$ & $a$\\
        $\times$ & $b$\\
        $\times$ & $b$\\
    \hline
    \end{tabular}
    \\[3.0em]
     & $6$ ways & $2$ ways
\end{tabular}
\vspace{-0.8em}
\end{GrayBox}
\end{table}    
\vspace{-1.5em}
\end{proof}

\begin{corollary}
$S^{4}_{4}(t) = S^{4/1}_{4}(t) + S^{4/2}_{4}(t) = \frac{m_1^4 t(1+7t+4t^2)}{(1-t)^5} e^{t \left(\mu_4-3\right)}$.
\end{corollary}
\begin{corollary}
$F_4(t) = \frac{e^{t(\mu_4-3)}}{\left(1-t\right)^3} \left(\left(1+m_1 \mu _3 t\right){}^4+6m_1^2 t\frac{\left(1+m_1 \mu _3t\right){}^2}{1-t}+m_1^4 t\frac{1+7t+4t^2}{\left(1-t\right){}^2}\right)$.
\end{corollary}
\end{proof}
\section{Sixth moment when $m_1 = 0$
}
The proof of the following theorem was already established in our paper \cite{LP2022}. In this section, we provide a more compact version of the proof based on inclusion/exclusion and the fact we know the EGF for the special case where $A_{ij}$ is normally distributed.

\begin{theorem}[Beck, Lv, Potechin 2023 \cite{LP2022}]\label{thm:F6cen}
For any distribution of $\Omega$ with $m_1=0$,
\begin{equation*}
    F_6(t)|_{m_1 = 0} = (1+m_3^2t)^{10}\,\frac{e^{t \left(m_6-15m_4m_2-10m_3^2 + 30m_2^3\right)}}{\left(1+3m_2^3t-m_4m_2t\right)^{15}} N_6\left(\frac{m_2^3t}{\left(1+3m_2^3t-m_4m_2t\right)^{3}}\right).
\end{equation*}
where $N_6(t) = \frac{1}{48}\sum_{n=0}^\infty (n+1)(n+2)(n+4)! \,t^n$ is the exponential generating function for the sixth moment of random matrices with i.i.d. Gaussian entries.
\end{theorem}
\begin{corollary}
Furthermore, defining $q_6 = m_6 - 10m_3^2- 15m_4m_2+ 30m_2^3$ and $q_4 = m_4m_2- 3m_2^3$,
\begin{equation*}
f_6(n)|_{m_1=0} = n!^2 \sum _{j=0}^n \sum _{i=0}^j \sum_{k=0}^{n-j} \frac{(1+i) (2+i) (4+i)! }{48 (n-j-k)!}\binom{10}{k}\binom{14+j+2i}{j-i} q_6^{n-j-k} q_4^{j-i}m_3^{2k} m_2^{3i}.
\end{equation*}
\end{corollary}

\begin{proof}
Without the loss of generality, we assume $m_2 = 1$ throughout the proof. The fact that $m_1 = 0$ reduces the number of tables with nontrivial weight. It is convenient to denote $\mathcal{F}^\mathrm{cen}_6$ as the set of those tables (where the column order is irrelevant), which in turn contribute to the sum $f_6(n)|_{m_1=0}$. These tables can be constructed out of the following columns (up to permutations  of the rows):

\vspace{-1em}
\bgroup
\renewcommand{\arraystretch}{0.7}
\begin{table}[H]
\centering
\setlength{\tabcolsep}{5.2pt}
\begin{GrayBox}
\vspace{-0.5em}
\centering
\begin{tabular}{ccccc}
    Type: & $6$-column & $4$-column & $2$-column & $3$-column \\[0.3em]
    $\mathcal{F}_6^{\mathrm{cen}}:$& \begin{tabular}{|>{\columncolor{mycolor}}c|}
    \hline
        $a$\\
        $a$\\
        $a$\\
        $a$\\
        $a$\\
        $a$\\
    \hline
    \end{tabular}
    &
        \begin{tabular}{|>{\columncolor{mycolor}}c|}
    \hline
        $a$\\
        $a$\\
        $a$\\
        $a$\\
        $b$\\
        $b$\\
    \hline
    \end{tabular}
    &
        \begin{tabular}{|>{\columncolor{mycolor}}c|}
    \hline
        $a$\\
        $a$\\
        $b$\\
        $b$\\
        $c$\\
        $c$\\
    \hline
    \end{tabular}
    &
        \begin{tabular}{|>{\columncolor{mycolor}}c|}
    \hline
        $a$\\
        $a$\\
        $a$\\
        $b$\\
        $b$\\
        $b$\\
    \hline
    \end{tabular}    \\[2.6em]
    Weight: & $m_6$ & $m_4$ & $1$ & $m_3^2$
\end{tabular}
\vspace{-0.6em}
\end{GrayBox}
\end{table}
\egroup
\vspace{-1em}

\noindent
When the entries of the matrix are $N(0,1)$, we have that $m_6 = 15$ and $m_4 = 3$. In this case, we can capture the weights by counting columns together with a pairing of equal elements in the column as there are $3$ such pairings for a 4-column and $15$ such pairings for a 6-column.

We handle general $m_4$ and $m_6$ using inclusion/exclusion. In particular, we pretend that each 4-column and 6-column is either \emph{known} or \emph{unknown} where the case when the column is unknown corresponds to having a pairing between elements in the column and the \emph{known} part captures the deviation of $m_4$ from $3$ and the deviation of $m_6$ from $15$. 

We can represent known and unknown columns visually as follows where the weights are carefully chosen to give a total of $m_4$ for 4-columns and a total of $m_6$ for 6-columns. Note that the elements $a'$, $b'$, and $c'$ in an \emph{unknown} column are \textbf{not necessarily distinct} but for a known 4-column, we require that $b' \neq a$ (if we allowed $b' = a$ then we would instead have a weight of $m_6 - 15m_4 + 30$ for known 6-columns).


\vspace{-1em}
\begin{table}[H]
\centering
\setlength{\tabcolsep}{5.2pt}
\begin{GrayBox}
\vspace{-0.7em}
\centering
\begin{tabular}{ccccc}
    Type: & \begin{tabular}{c} known\\[-0.2em] $6$-column \end{tabular} & \begin{tabular}{c} known\\[-0.2em] $4$-column \end{tabular} & $3$-column & \begin{tabular}{c} unknown\\[-0.2em] column \end{tabular}   \\[0.8em]
    $\mathcal{F}_6^*:$&         \begin{tikzpicture}[baseline = 9.2ex,scale = 0.45]
        \draw[fill=mycolor] (1.3-0.6, 0.5) rectangle (1*1.3+0.6, 6.5);
        \pairA[1.3];
        \draw[fill=gray!0] (1*1.3-0.35, 0.8) rectangle (1*1.3+0.35, 6.2);
        \node[vertex] (a) at (1*1.3,3.55) {$a$};
        \end{tikzpicture}
    &
        \begin{tikzpicture}[baseline = 9.2ex,scale = 0.45]
        \draw[fill=mycolor] (1.3-0.6, 0.5) rectangle (1*1.3+0.6, 6.5);
        \pairA[1.3];
        \draw[fill=gray!0] (1*1.3-0.35, 2.6) rectangle (1*1.3+0.35, 6.2);
        \node[vertex] (a) at (1*1.3,4.5) {$a$};
        \draw[pair] (A1.center) to node[fill=mycolor, inner sep=1.0pt] {$b'$} (A2.center);
        \end{tikzpicture}
    &
        \begin{tikzpicture}[baseline = 9.2ex,scale = 0.45]
        \draw[fill=mycolor] (1.2-.6, 0.5) rectangle (1*1.2+.6, 6.5);
        \node[vertex] (A1) at (1*1.2,1*0.93+0.3) {$b$};
        \node[vertex] (A2) at (1*1.2,2*0.93+0.3) {$b$};
        \node[vertex] (A3) at (1*1.2,3*0.93+0.3) {$b$};
        \node[vertex] (A4) at (1*1.2,4*0.93+0.3) {$a$};
        \node[vertex] (A5) at (1*1.2,5*0.93+0.3) {$a$};
        \node[vertex] (A6) at (1*1.2,6*0.93+0.3) {$a$};
        \end{tikzpicture}
    &
        \begin{tikzpicture}[baseline = 9.2ex,scale = 0.45]
        \draw[fill=mycolor] (1.3-0.6, 0.5) rectangle (1*1.3+0.6, 6.5);
        \pairA[1.3];
        \draw[pair] (A1.center) to [bend right=45] node[inner sep=1.0pt] {\,\,\,\,\,\,\,$b'$} (A5.center);
        \draw[pair] (A3.center) to [bend left=33] node[inner sep=1.0pt] {$a'$\,\,\,\,\,\,\,\,} (A6.center);
        \draw[pair] (A2.center) to [bend left=45] node[inner sep=1.0pt] {$c'$\,\,\,\,\,\,\,\,\,} (A4.center);
        \end{tikzpicture}
    \\[3.9em]
    Weight: & $m_6\!-\!15$ & $m_4\!-\!3$ & $m_3^2$ & $1$
\end{tabular}
\vspace{-0.5em}
\end{GrayBox}
\end{table}    
\vspace{-1em}

Figure \ref{fig:QuadriEx0} shows an example of a table $\tau \in \mathcal{F}^*_{6,9}$ with two known $4$-columns (each with weight $m_4-3$) and one known $6$-column. Note that since a $4$-column 
can either be known (weight $m_4-3$) or unknown (there are $3$ ways how we can pair up the four identical elements), the total contribution 
is $m_4 - 3 + 3 = m_4$, which matches the contribution of a $4$-column to $\mathcal{F}^\mathrm{cen}_6$. This decomposition is shown in Figure \ref{fig:exclinclF6}.

\begin{figure}[tbh!]
\centering
\begin{minipage}{.48\textwidth}
    \centering
        \begin{tikzpicture}[baseline = 8.2ex,scale = 0.50]
        \pairABCDEFGHI[1.3];
        \draw (1.3-1, 0.35) rectangle (9*1.3+1, 6.65);
        \draw[fill=mycolor,draw=none] (1.3-1, 0.5) rectangle (9*1.3+1, 6.5);
        \draw[fill=white] (2*1.3-0.3, 0.8) rectangle (2*1.3+0.3, 6.2);
        \node[vertex] (3) at (2*1.3,3.5) {3};
        \draw[fill=white] (3*1.3-0.3, 0.8) rectangle (3*1.3+0.3, 4.2);
        \node[vertex] (1) at (3*1.3,2.5) {1};
        \draw[fill=white] (7*1.3-0.3, 2.6) rectangle (7*1.3+0.3, 6.2);
        \node[vertex] (9) at (7*1.3,4.5) {9};
        \draw[pair] (A1.center) to [bend left=25] node[fill=mycolor, inner sep=1.0pt] {4} (A6.center);
        \draw[pair] (A2.center) to node[fill=mycolor, inner sep=1.0pt] {4} (A3.center);
        \draw[pair] (A4.center) to node[fill=mycolor, inner sep=1.0pt] {4} (A5.center);
        \draw[pair] (C5.center) to node[fill=mycolor, inner sep=1.0pt] {7} (C6.center);
        \draw[pair] (D1.center) to node[fill=mycolor, inner sep=1.0pt] {5} (D2.center);
        \draw[pair] (D3.center) to [bend left=45] node[fill=mycolor, inner sep=1.0pt] {2} (D6.center);
        \draw[pair] (D4.center) to node[fill=mycolor, inner sep=1.0pt] {2} (D5.center);
        \draw[pair] (E1.center) to node[fill=mycolor, inner sep=1.0pt] {6} (E2.center);
        \draw[pair] (E3.center) to [bend left=-45] node[fill=mycolor, inner sep=1.0pt] {5} (E6.center);
        \draw[pair] (E4.center) to node[fill=mycolor, inner sep=1.0pt] {8} (E5.center);
        \draw[pair] (F1.center) to [bend left=45] node[fill=mycolor, inner sep=1.0pt] {7} (F4.center);
        \draw[pair] (F2.center) to node[fill=mycolor, inner sep=1.0pt] {7} (F3.center);
        \draw[pair] (F5.center) to node[fill=mycolor, inner sep=1.0pt] {1} (F6.center);
        \draw[pair] (G1.center) to node[fill=mycolor, inner sep=1.0pt] {8} (G2.center);
        \draw[pair] (H1.center) to node[fill=mycolor, inner sep=1.0pt] {9} (H2.center);
        \draw[pair] (H3.center) to [bend left=45] node[fill=mycolor, inner sep=1.0pt] {6} (H6.center);
        \draw[pair] (H4.center) to node[fill=mycolor, inner sep=1.0pt] {5} (H5.center);
        \draw[pair] (I1.center) to node[fill=mycolor, inner sep=1.0pt] {2} (I2.center);
        \draw[pair] (I3.center) to [bend left=-45] node[fill=mycolor, inner sep=1.0pt] {8} (I6.center);
        \draw[pair] (I4.center) to node[fill=mycolor, inner sep=1.0pt] {6} (I5.center);
        \draw (1.3-1, 0.5) rectangle (9*1.3+1, 6.5);
        \end{tikzpicture}
\caption{A table $\tau \in \mathcal{F}^*_{6,9}$ with weight $w(\tau) = (m_6-15)(m_4-3)^2$.}
\label{fig:QuadriEx0}
    \end{minipage}%
\hfill
\begin{minipage}{.50\textwidth}
    \centering
\begin{tabular}{ccccccccc}
    \begin{tikzpicture}[baseline = 8.2ex,scale = 0.42]
        \draw[fill=mycolor] (1.2-.6, 0.5) rectangle (1 *1.2+.6, 6.5);
        \node[vertex] (A1) at (1*1.2,1) {$b$};
        \node[vertex] (A2) at (1*1.2,2) {$b$};
        \node[vertex] (A3) at (1*1.2,3) {$a$};
        \node[vertex] (A4) at (1*1.2,4) {$a$};
        \node[vertex] (A5) at (1*1.2,5) {$a$};
        \node[vertex] (A6) at (1*1.2,6) {$a$};
        \end{tikzpicture}
    & \!\!$=$\!\! &
        \begin{tikzpicture}[baseline = 8.2ex,scale = 0.42]
        \draw[fill=mycolor] (1.2-.6, 0.5) rectangle (1 *1.2+.6, 6.5);
        \pairA[1.2];
        \draw[fill=white] (1*1.2-0.3, 2.6) rectangle (1*1.2+0.3, 6.2);
        \node[vertex] (a) at (1*1.2,4.5) {$a$};
        \draw[pair] (A1.center) to node[fill=mycolor, inner sep=1.0pt] {$b$} (A2.center);
        \end{tikzpicture}
    & \!\!\!\!\!$+$\!\! &
        \begin{tikzpicture}[baseline = 8.2ex,scale = 0.42]
        \draw[fill=mycolor] (1.2-.5, 0.5) rectangle (1 *1.2+.5, 6.5);
        \pairA[1.2];
        \draw[pair] (A1.center) to node[fill=mycolor, inner sep=1.0pt] {$b$} (A2.center);
        \draw[pair] (A3.center) to node[fill=mycolor, inner sep=1.0pt] {$a$} (A4.center);
        \draw[pair] (A5.center) to node[fill=mycolor, inner sep=1.0pt] {$a$} (A6.center);
        \end{tikzpicture}
    & \!\!$+$\!\! &
        \begin{tikzpicture}[baseline = 8.2ex,scale = 0.42]
        \draw[fill=mycolor] (1.2-.7, 0.5) rectangle (1 *1.2+.5, 6.5);
        \pairA[1.2];
        \draw[pair] (A1.center) to node[fill=mycolor, inner sep=1.0pt] {$b$} (A2.center);
        \draw[pair] (A3.center) to [bend left=34] node[fill=mycolor, inner sep=1.0pt] {$a$} (A6.center);
        \draw[pair] (A4.center) to node[fill=mycolor, inner sep=1.0pt] {$a$} (A5.center);
        \draw (1.2-.7, 0.5) rectangle (1 *1.2+.5, 6.5);
        \end{tikzpicture}
    & \!\!$+$\!\! &
        \begin{tikzpicture}[baseline = 8.2ex,scale = 0.42]
        \draw[fill=mycolor,draw=none] (1.2-.7, 0.5) rectangle (1 *1.2+.7, 6.5);
        \pairA[1.2];
        \draw[pair] (A1.center) to node[fill=mycolor, inner sep=1.0pt] {$b$} (A2.center);
        \draw[pair] (A3.center) to [bend right=54] node[fill=mycolor, inner sep=1.0pt] {$a$} (A5.center);
        \draw[pair] (A4.center) to [bend left=54] node[fill=mycolor, inner sep=1.0pt] {$a$} (A6.center);
        \draw (1.2-.7, 0.5) rectangle (1 *1.2+.7, 6.5);
        \end{tikzpicture}
    \\[2.5em]
    $m_4$ & & $m_4\!-\!3$ & & $1$ &  & $1$ & & $1$
\end{tabular}
    \caption{Inclusion/exclusion for $4$-columns}
    \label{fig:exclinclF6}
    \end{minipage}%
\end{figure}
\vspace{-0.5em}

\FloatBarrier
\noindent
A similar analysis holds for the 6-columns. Overall, we obtain that 
\begin{equation}\label{Eq:f6matching}
    f_6(n)|_{m_1,m_3=0} = n! \sum_{\tau \in \mathcal{F}_6^\mathrm{cen}} w(\tau) \sgn(\tau) = n! \sum_{\tau \in \mathcal{F}_6^*} w(\tau) \sgn(\tau).
\end{equation}

\noindent
The key idea is that permutation tables $\mathcal{F}_6^*$ can be decomposed into several components. 
\begin{enumerate}
\item Known $6$-columns, composed as $\textsc{Set}\big{(}(m_6 - 15) \,\, \oneincirc \,\big{)}$ with EGF equal to
\begin{equation*}
e^{(m_6-15)t}.
\end{equation*}
\item Cycles of known $4$-columns, or $\textsc{Set}\big{(} 15\,\textsc{Cyc}_{\geq 2}\big{(}(m_4-3) \, \oneincirc \big{)}\big{)}$ with EGF
\begin{equation*}
\frac{e^{-15(m_4-3)t}}{(1-(m_4-3)t)^{15}}
\end{equation*}
\item Cycles of $3$-columns, or $\textsc{Set}\big{(}-10\,\textsc{Cyc}_{\geq 2}\big{(}-m_3^2 \,\, \oneincirc \big{)}\big{)}$ with EGF
\begin{equation*}
(1+ m_3^2t)^{10} e^{-10m_3^2t}
\end{equation*}
\item A ``core'' of unknown columns together with paths of known $4$-columns leading to pairs in the core. We can analyze this part as follows. Letting $\mathcal{N}_6$ be the structure for the Gaussian case, $\mathcal{N}_6$ is also the structure for the core without the attached paths of known $4$-columns with EGF being the known $N_6(t) = \frac{1}{48}\sum_{n=0}^\infty (n+1)(n+2)(n+4)! \,t^n$. We now observe that each column of the core has three paths (possibly of length $0$) of known $4$-columns leading to it. Structurally, this gives
\begin{equation}
    \mathcal{N}_6 \left(
    \begin{tikzpicture}[baseline = 6.2ex,scale = 0.41]
        \draw[fill=mycolor] (1.3-0.5, 0.5) rectangle (1*1.3+0.5, 5.5);
        \node[vertex] (A1) at (1*1.3,1) {};
        \node[vertex] (A2) at (1*1.3,1.7) {};
        \node[vertex] (A3) at (1*1.3,2.7) {};
        \node[vertex] (A4) at (1*1.3,3.3) {};
        \node[vertex] (A5) at (1*1.3,4.3) {};
        \node[vertex] (A6) at (1*1.3,5) {};
        \draw[pair] (A5.center) to (A6.center);
        \draw[pair] (A3.center) to (A4.center);
        \draw[pair] (A1.center) to (A2.center);
    \end{tikzpicture}
        \star \,
    \textsc{Seq}\left(\,\fourcolab\,\right)
    \star
    \textsc{Seq}\left(\,\fourcolabII\,\right)
    \star
    \textsc{Seq}\left(\,\fourcolabIII\,\right)
    \right).
\end{equation}
Since the EGF for $\textsc{Seq}\big{(}(m_4-3) \,\, \oneincirc \,\big{)}$ is $\frac{1}{1 - (m_4 - 3)t}$, the EGF for the core and the paths of known $4$-columns leading to it is $N_6\big{(}\frac{t}{(1-(m_4-3)t)^3}\big{)}$.
\end{enumerate}

\begin{figure}[H]
    \centering
    \begin{minipage}{.49\textwidth}
        \centering
\begin{tikzpicture}[baseline = 8.2ex,scale = 0.51]
        \draw (1*1.2+.6, 0.35) -- (1.2-.6, 0.35) -- (1.2-.6, 6.65) -- (1*1.2+.6, 6.65);
        \draw[fill=mycolor] (1.2-.6, 0.5) rectangle (1*1.2+.6, 6.5);
        \pairA[1.2];
        \draw[fill=gray!0] (1*1.2-0.3, 0.9) rectangle (1*1.2+0.3, 6.2);
        \node[vertex] (11) at (1*1.2,3.5) {11};
        \end{tikzpicture}
        \begin{tikzpicture}[baseline = 8.2ex,scale = 0.51]
        \draw (2*1.2+.6, 0.35) -- (1.2-.6, 0.35);
        \draw (1.2-.6, 6.65) -- (2*1.2+.6, 6.65);
        \draw[fill=mycolor] (1.2-.6, 0.5) rectangle (2*1.2+.6, 6.5);
        \pairAB[1.2];
        \draw[fill=gray!0] (1*1.2-0.3, 2.6) rectangle (1*1.2+0.3, 6.2);
        \node[vertex] (9) at (1*1.2,4.5) {9};
        \draw[fill=gray!0] (2*1.2-0.3, 2.6) rectangle (2*1.2+0.3, 6.2);
        \node[vertex] (10) at (2*1.2,4.5) {10};
        \draw[pair] (A1.center) to node[fill=mycolor, inner sep=1.0pt] {10} (A2.center);
        \draw[pair] (B1.center) to node[fill=mycolor, inner sep=1.0pt] {9} (B2.center);
        \end{tikzpicture}
        \begin{tikzpicture}[baseline = 8.2ex,scale = 0.51]
        \draw (1.2-.6, 0.35) -- (3*1.2+.6, 0.35);
        \draw (1.2-.6, 6.65) -- (3*1.2+.6, 6.65);
        \draw[fill=mycolor] (1.2-.6, 0.5) rectangle (3*1.2+.6, 6.5);
        \pairABC[1.2];
        \draw[fill=gray!0] (1*1.2-0.3, 0.8) rectangle (1*1.2+0.3, 4.2);
        \node[vertex] (2) at (1*1.2,2.5) {2};
        \draw[fill=gray!0] (2*1.2-0.3, 0.8) rectangle (2*1.2+0.3, 4.2);
        \node[vertex] (3) at (2*1.2,2.5) {3};
        \draw[fill=gray!0] (3*1.2-0.3, 0.8) rectangle (3*1.2+0.3, 4.2);
        \node[vertex] (7) at (3*1.2,2.5) {7};
        \draw[pair] (A5.center) to node[fill=mycolor, inner sep=1.0pt] {7} (A6.center);
        \draw[pair] (B5.center) to node[fill=mycolor, inner sep=1.0pt] {2} (B6.center);
        \draw[pair] (C5.center) to node[fill=mycolor, inner sep=1.0pt] {3} (C6.center);
        \end{tikzpicture}       
        \begin{tikzpicture}[baseline = 8.2ex,scale = 0.51]
        \draw (1.3-1, 6.65) -- (5*1.3+1, 6.65) -- (5*1.3+1, 0.35) -- (1.3-1, 0.35);
        \draw[fill=mycolor, draw=none] (1.3-1, 0.5) rectangle (5*1.3+1, 6.5);
        \draw[fill=black!10, draw=none] (3*1.3-0.7, 0.5) rectangle (4*1.3+0.7, 6.5);
        \pairABCDE[1.3];
        \draw[fill=gray!0] (1*1.3-0.2, 1.5) rectangle (1*1.3+0.4, 5.5);
        \node[vertex] (4) at (1*1.3+0.1,2.7) {4};
        \draw[fill=gray!0] (2*1.3-0.4, 1.5) rectangle (2*1.3+0.2, 5.5);
        \node[vertex] (6) at (2*1.3-0.1,2.7) {6};
        \draw[fill=gray!0] (5*1.3-0.3, 0.8) rectangle (5*1.3+0.3, 1.4);
        \draw[fill=gray!0] (5*1.3-0.3, 3.6) rectangle (5*1.3+0.3, 6.2);
        \node[vertex] (5) at (5*1.3,4.5) {5};
        \draw[pair] (A1.center) to [bend left=30] node[fill=mycolor, inner sep=1.0pt] {1} (A6.center);
        \draw[pair] (B1.center) to [bend right=30] node[fill=mycolor, inner sep=1.0pt] {4} (B6.center);
        \draw[pair] (C1.center) to [bend right=25] node[fill=black!10, inner sep=1.0pt] {6} (C6.center);
        \draw[pair] (C2.center) to node[fill=black!10, inner sep=1.0pt] {8} (C3.center);
        \draw[pair] (C4.center) to node[fill=black!10, inner sep=1.0pt] {1} (C5.center);
        \draw[pair] (D1.center) to [bend right=45] node[fill=black!10, inner sep=1.0pt] {8} (D4.center);
        \draw[pair] (D2.center) to node[fill=black!10, inner sep=1.0pt] {5} (D3.center);
        \draw[pair] (D5.center) to node[fill=black!10, inner sep=1.0pt] {8} (D6.center);
        \draw[pair] (E2.center) to node[fill=mycolor, inner sep=1.0pt] {1} (E3.center);
        \draw (1.3-1, 0.5) rectangle (5*1.3+1, 6.5);
        \end{tikzpicture}
    \end{minipage}%
    \hfill
    \begin{minipage}{.5\textwidth}
        \centering
        \begin{tikzpicture}[thick, main/.style = {draw,circle, inner sep = 2pt}, scale = 0.8]
        \node[main] (9) at (-4.2,1.4) {$9$};
        \node[main] (10) at (-4.2,-0.7) {$10$};
        \node[main] (2) at (-0.4,1.4) {$2$};
        \node[main] (3) at (-3,1.4) {$3$};
        \node[main] (7) at (-1.7,-0.7) {$7$};
        \node[main] (v) at (0.4,0) {$\nu$};
        \node[main] (6) at (1.7,0.7) {$6$};
        \node[main] (5) at (1.7,-0.7) {$5$};
        \node[main] (4) at (3,1.4) {$4$};
        \draw[->] (9) to [bend right=40] (10);
        \draw[->] (10) to [bend right=40] (9);
        \draw[->] (2) to [bend right=40] (3);
        \draw[->] (3) to [bend right=40] (7);
        \draw[->] (7) to [bend right=40] (2);
        \draw[->] (6) to [bend right=0] (v);
        \draw[->] (4) to [bend right=0] (6);
        \draw[->] (5) to [bend right=0] (v);
        \end{tikzpicture}
    \end{minipage}%
\caption{An $\mathcal{F}^*_{6,11}$ table with disjoint components (known 6-columns, cycles of known 4-columns, cycles of 6-columns and the core $\nu$ in gray with attached paths of known 4-columns) and the corresponding graph of dependencies of known $4$-columns}
\label{fig:graphF6star}
\end{figure}
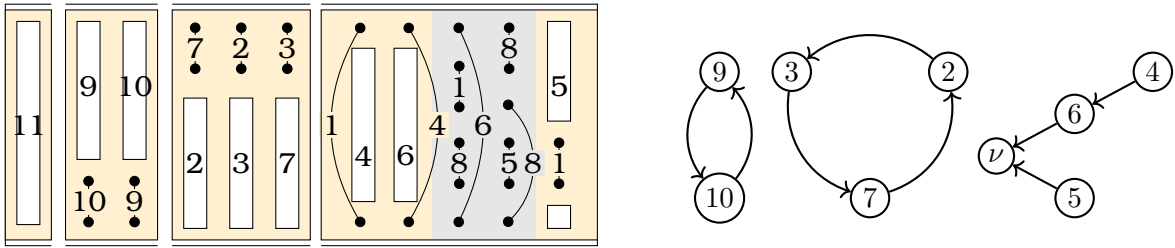

\noindent
By joining all disjoint components (see the example in Figure \ref{fig:graphF6star}), we get, in terms of EGF's,

\begin{equation*}
F_6(t)|_{m_1 = 0} \!=\! \underbrace{e^{(m_6-15)t}}_{\text{known $6$-columns}} \underbrace{\frac{e^{-15(m_4\!-3)t}}{(1-(m_4-3)t)^{15}}}_{\text{cycles of known $4$-columns}} \underbrace{(1\!+\! m_3^2t)^{10} e^{-10m_3^2t}}_{\text{cycles of $3$-columns}} \,\,N_6\!\!\!\!\underbrace{\left(\!\frac{t}{(1-(m_4-3)t)^3}\right)}_{\text{$4$-columns attached to the core}}
\end{equation*}

\noindent
The function $N_6(t)$ corresponds to the exponential generating function of tables containing only unknown columns. By inclusion/exclusion, $N_6(t)$ equals $F_6(t)$ assuming $m_1=m_3=m_5=0$ and $m_2=1,m_4=3,m_6=15$. Equivalently, $N_6(t)$ is the sixth moment of the determinant of a matrix whose entries are i.i.d. with the standard normal distribution. By the known result for the even moments of the determinant of such a matrix (Pr\'{e}kopa \cite{prekopa1967random}), we have that 
\begin{equation*}
    F_6(t)|_{X_{ij} \,\sim\, \mathsf{N}(0,1)} = N_6(t) = \frac{1}{48}\sum_{n=0}^\infty (n+1)(n+2)(n+4)! \,t^n.
\end{equation*}

\end{proof}

\printbibliography

\end{document}